\documentclass[12pt,a4paper]{amsart}
\usepackage{amssymb,amsxtra}
\usepackage[cmtip,arrow]{xy}
\usepackage{pb-diagram,pb-xy}

\calclayout

\DeclareMathOperator{\Tor}{Tor}
\DeclareMathOperator{\Ext}{Ext}
\DeclareMathOperator{\Hom}{Hom}
\DeclareMathOperator{\Nilp}{Nilp}
\DeclareMathOperator{\Cob}{Cob}
\DeclareMathOperator{\Br}{Bar}
\DeclareMathOperator{\id}{id}
\DeclareMathOperator{\im}{im}
\DeclareMathOperator{\gr}{gr}
\DeclareMathOperator{\q}{q}

\DeclareMathSymbol{\dabar}{\mathord}{AMSa}{"39}
\DeclareMathSymbol{\dahead}{\mathord}{AMSa}{"4B}

\renewcommand{\le}{\leqslant}
\renewcommand{\ge}{\geqslant}

\newcommand{\rarrow}{\longrightarrow}
\newcommand{\ot}{\otimes}

\newcommand{\+}{\nobreakdash-}
\renewcommand{\;}{,\>}
\renewcommand{\.}{\mskip.5\thinmuskip}
\renewcommand{\:}{\colon}

\renewcommand{\d}{\partial}

\newcommand{\bu}{{\text{\smaller\smaller$\scriptstyle\bullet$}}}
\newcommand{\subbu}{{\text{\smaller\smaller$\scriptscriptstyle\bullet$}}}
\newcommand{\lrarrow}{\.\relbar\joinrel\relbar\joinrel\rightarrow\.}
\newcommand{\ldarrow}{\mathrel{\.\dabar\dabar\dabar\dahead\.}}
\newcommand{\comp}{\sphat\,\.}

\renewcommand{\b}{{\mathsf b}}

\newcommand{\sC}{{\mathsf C}}
\newcommand{\sD}{{\mathsf D}}
\newcommand{\sF}{{\mathsf F}}
\newcommand{\sG}{{\mathsf G}}

\newcommand{\Z}{{\mathbb Z}}

\newcommand{\rop}{{\mathrm{op}}}

\newcommand{\modl}{{\operatorname{\mathsf{--mod}}}}
\newcommand{\comodl}{{\operatorname{\mathsf{--comod}}}}

\newcommand{\wot}{\mathbin{\widehat\otimes}}

\newcommand{\Section}[1]{\bigskip\section{#1}\medskip}
\theoremstyle{plain}
\newtheorem{thm}{Theorem}[section]
\newtheorem{lem}[thm]{Lemma}
\newtheorem{prop}[thm]{Proposition}
\newtheorem{cor}[thm]{Corollary}
\theoremstyle{definition}
\newtheorem{rem}[thm]{Remark}
\newtheorem{ex}[thm]{Example}

\begin{document}


\title{Koszulity of cohomology \\ $=$ $K(\pi,1)$-ness $+$
quasi-formality}

\author{Leonid Positselski}

\address{Department of Mathematics, Faculty of Natural Sciences,
University of Haifa, Mount Carmel, Haifa 31905, Israel; and
\newline\indent Laboratory of Algebraic Geometry, National Research
University Higher School of Economics, Moscow 117312; and
\newline\indent Sector of Algebra and Number Theory, Institute for
Information Transmission Problems, Moscow 127051, Russia}

\email{posic@mccme.ru}

\begin{abstract}
 This paper is a greatly expanded version
of~\cite[Section~9.11]{Partin}.
 A series of definitions and results illustrating the thesis in
the title (where quasi-formality means vanishing of a certain kind
of Massey multiplications in the cohomology) is presented.
 In particular, we include a categorical interpretation of
the ``Koszulity implies $K(\pi,1)$'' claim, discuss
the differences between two versions of Massey operations, and
apply the derived nonhomogeneous Koszul duality theory in order
to deduce the main theorem.
 In the end we demonstrate a counterexample providing a negative
answer to a question of Hopkins and Wickelgren about formality of
the cochain DG-algebras of absolute Galois groups, thus showing
that quasi-formality cannot be strengthened to formality in
the title assertion.
\end{abstract}

\maketitle

\tableofcontents

\section*{Introduction}
\medskip

 A \emph{quadratic algebra} is an associative algebra defined by
homogeneous quadratic relations.
 In other words, a positively graded algebra
$A=k\oplus A_1\oplus A_2\oplus\dotsb$ over a field~$k$ is called
quadratic if it is generated by its first-degree component $A_1$
with relations in degree~$2$.
 A positively graded associative algebra $A$ is called
\emph{Koszul}~\cite{Pr,BGSoe,PV} if one has $\Tor^A_{ij}(k,k)=0$
for all $i\ne j$, where the first grading~$i$ on the Tor spaces is
the usual homological grading and the second grading~$j$, called
the \emph{internal} grading, is induced by the grading of~$A$.
 In particular, this condition for $i=1$ means that the algebra $A$
is generated by $A_1$, and the conditions for $i=1$ and~$2$ taken
together mean that $A$ is quadratic.

 Conversely, for any positively graded algebra $A$ with
finite-dimensional components $A_n$ the diagonal part
$\bigoplus_n\Ext_A^{n,n}(k,k)$ of the algebra $\Ext_A^*(k,k)\simeq
\Tor^A_*(k,k)^*$ is a quadratic algebra.
 When the algebra $A$ is quadratic, the two quadratic algebras
$A$ and $\bigoplus_n\Ext_A^{n,n}(k,k)$ are called \emph{quadratic
dual} to each other.
 Without the finite-dimensionality assumption on the grading
components, the quadratic duality connects quadratic graded
algebras with quadratic graded coalgebras~\cite{PV,Pbogom}.

 When one attempts to deform quadratic algebras by considering algebras
with \emph{nonhomogeneous} quadratic relations, one discovers that
there are two essentially different ways of doing so.
 One can either consider relations with terms of the degrees not
greater than~$2$, that is
\begin{equation} \label{210}
  q_2(x)+q_1(x)+q_0=0,
\end{equation}
where $x=(x_\alpha)$, \ $\deg x_\alpha=1$ denotes the set of
generators, and $\deg q_n=n$; or relations with terms of
the degrees not less than~$2$, that is
\begin{equation} \label{234}
  q_2(x)+q_3(x)+q_4(x)+q_5(x)+\dotsb=0.
\end{equation}

 In the latter case, it is natural to allow the relations to be
infinite power series, that is consider the algebra they define
as a quotient algebra of the algebra of formal Taylor power series
in noncommuting variables~$x_\alpha$, rather than the algebra of
noncommutative polynomials.
 Almost equivalently, this means considering a set of relations of
the type~\eqref{234} as defining a conilpotent \emph{coalgebra},
while a set of relations of the type~\eqref{210} defines
a filtered \emph{algebra}.
 More precisely, of course, one should say that the complete topological
algebra defined by the relations~\eqref{234} is the dual vector space
to a discrete conilpotent coalgebra.
 This is one of the simplest ways to explain the importance of
coalgebras in Koszul duality.

 The alternative between considering nonhomogeneous quadratic
relations of the types~\eqref{210} and~\eqref{234} roughly leads
to a division of the Koszul duality theory into two streams,
the former of them going back to the classical paper~\cite{Pr}
and the present author's work~\cite{Pcurv}, and the latter one
originating in the paper~\cite{PV}.
 The former theory, invented originally for the purposes of computing
the cohomology of associative algebras generally and the Steenrod
algebra in particular, eventually found its applications in
semi-infinite homological algebra.
 The latter point of view was being applied to Galois cohomology,
the conjectures about absolute Galois groups, and the theory of motives
with finite coefficients.
 As a general rule, the author's subsequent papers vaguely associated
with relations of the type~\eqref{210} were published on
the \texttt{arXiv} in the subject area \texttt{[math.CT]}, while
the papers having to do with relations of the type~\eqref{234} were
put into the area \texttt{[math.KT]}.

 This paper is concerned with relations of the type~\eqref{234}.
 Its subject can be roughly described as cohomological
characterization of the coalgebras $C$ defined by
the relations~\eqref{234} with the quadratic principal parts
\begin{equation} \label{homog2}
  q_2(x)=0
\end{equation}
of the relations defining a Koszul graded coalgebra.
 In fact, according to the main theorem of~\cite{PV}
(see also~\cite{Lee}) a conilpotent coalgebra $C$ is defined by
a self-consistent system of relations~\eqref{234} with Koszul
quadratic principal part~\eqref{homog2} if and only if its
cohomology algebra
$$
 H^*(C)=\Ext^*_C(k,k)
$$
is Koszul.
 Moreover, a certain seemingly weaker set of conditions on the algebra
$H^*=H^*(C)$ is sufficient, and implies Koszulity of algebras of
the form $H^*(C)$.
 Specifically, the algebra $H^*(C)$ is Koszul whenever $H^2(C)$ is
multiplicatively generated by $H^1(C)$, there are no nontrivial
relations of degree~$3$ between elements of degree~$1$ in $H^*(C)$,
and the quadratic part of the algebra $H^*(C)$ is Koszul
(see Theorem~\ref{with-vishik-main-theorem}).
 When the algebra $H^*(C)$ is Koszul, it is simply the dual quadratic
algebra to the quadratic coalgebra defined by
the relations~\eqref{homog2}.

 Given an arbitrary (not necessarily conilpotent) coaugmented
coalgebra~$D$ with the maximal conilpotent subcoalgebra $C=\Nilp D
\subset D$, the algebra $H^*(D)=\Ext_D^*(k,k)$ is Koszul if and only if
the following two conditions hold~\cite[Section~9.11]{Partin}:
\begin{enumerate}
\renewcommand{\theenumi}{\roman{enumi}}
\item the homomorphism of cohomology algebras $H^*(C)\rarrow H^*(D)$
induced by the embedding of coalgebras $C\rarrow D$ is an isomorphism;
\item a certain family of higher Massey products in the cohomology
algebra $H^*(D)$ vanishes.
\end{enumerate}
 In this paper we provide a detailed proof of this result, and
discuss at length its constituting components.

 In particular, the condition~(i) and the implication ``Koszulity of
$H^*(D)$ implies~(i)'' allow numerous analogues and generalizations,
including such assertions as
\begin{itemize}
\item for any discrete group $\Gamma$ whose cohomology algebra
      $H^*(\Gamma,k)$ with coefficients in a field~$k$ is Koszul,
      the cohomology algebra $H^*(\Gamma_k\sphat\.,\.k)$ of
      the $k$\+completion of the group $\Gamma$ is isomorphic to
      the algebra $H^*(\Gamma,k)$ \cite[Section~5]{Pbogom}; or
\item any rational homotopy type $X$ with a Koszul cohomology
      algebra $H^*(X,\mathbb Q)$ is a rational $K(\pi,1)$
      space~\cite{PY}.
\end{itemize}
 That is why we call the condition~(i) ``the $K(\pi,1)$ condition''.

 More generally, in place of the cochain DG\+algebra of a coaugmented
coalgebra $D$, consider an arbitrary nonnegatively cohomologically
graded augmented DG\+algebra $0\rarrow A^0\rarrow A^1\rarrow A^2\rarrow
\dotsb$ over a field~$k$ with $H^0(A^\bu)\simeq k$.
 Then the cohomology algebra $H^*(A^\bu)$ of the DG\+algebra $A^\bu$
is Koszul if and only if the following two conditions hold:
\begin{enumerate}
\renewcommand{\theenumi}{\roman{enumi}}
\item the cohomology coalgebra of the bar-construction of
      the augmented DG\+alge\-bra $A^\bu$ is concentrated in
      cohomological degree~$0$;
\item a certain family of higher Massey products in the cohomology
      algebra $H^*(A^\bu)$ vanishes.
\end{enumerate}
 Once again, we call the condition~(i) ``the $K(\pi,1)$ condition''.

 Furthermore, relation sets of the type~\eqref{234} are naturally
viewed up to variable changes
\begin{equation}  \label{123}
 x_\alpha\rarrow x_\alpha + p_{2,\.\alpha}(x) + p_{3,\.\alpha}(x)
 + p_{4,\.\alpha}(x) + \dotsb,
\end{equation}
where $\deg p_{n,\.\alpha}=n$.
 Some apparently different systems of relations~\eqref{234} may be also
equivalent, i.~e., they may mutually imply each other, or in other
words, generate the same closed ideal in the algebra
of noncommutative Taylor power series
(see Example~\ref{counter-equivalent-relations} for an illustration).
 A natural question is whether or when a system of relations~\eqref{234}
can be homogenized, i.~e., transformed into a system of relations
equivalent to~\eqref{homog2} by a variable change~\eqref{123}.
 We show that a variable change~\eqref{123} homogenizing a given system
of relations~\eqref{234} defining a conilpotent coalgebra~$C$ with
Koszul cohomology algebra $H^*(C)$ exists if and only if the cochain
DG\+algebra computing $H^*(C)$ is \emph{formal}, i.~e., can be connected
with its cohomology algebra by a chain of multiplicative
quasi-isomorphisms.

 Obviously, formality implies the Massey product vanishing
condition~(ii), which we accordingly call the \emph{quasi-formality}
condition.
 Not distinguishing formality from quasi-formality seems to be
a common misconception.
 The above explanations suggest that the cochain DG\+algebras of most
conilpotent coalgebras $C$ with Koszul cohomology algebras $H^*(C)$
should not be formal but only quasi-formal, as the possibility of
homogenizing a system of relations~\eqref{234} looks unlikely,
generally speaking.

 Indeed, we provide a simple counterexample of a pro-$l$-group $H$
whose cohomology algebra $H^*(H,\Z/l)$ is Koszul, while the cochain
DG\+algebra computing it is not formal, as the relations in the group
coalgebra $\Z/l(H)$ cannot be homogenized by variable changes. 
 It was conjectured in the papers~\cite{PV,Partin} that the cohomology
algebra $H^*(G_F,\Z/l)$ is Koszul for the absolute Galois group
$G_F$ of any field $F$ containing a primitive $l$\+root of unity;
and the question was asked in the paper~\cite{HW} whether the cochain
DG\+algebra of the group $G_F$ with coefficients in $\Z/l$ is formal.
 As the group $H$ in our counterexample is the maximal quotient
pro-$l$-group of the absolute Galois group $G_F$ of an appropriate
$p$\+adic field $F$ containing a primitive $l$\+root of unity, our
results provide a negative answer to this question of Hopkins and
Wickelgren.

\subsection*{Acknowledgements}
 The mathematical content of this paper was largely worked out around
the years 1996--97, when the author was a graduate student at
Harvard University.
 I benefited from helpful conversations with A.~Beilinson,
J.~Bernstein, and D.~Kazhdan at that time.
 Section~9.11, where these ideas were first put in writing, was
inserted into the paper~\cite{Partin} shortly before its journal
publication under the influence of a question asked by P.~Deligne
during the author's talk at the ``Christmas mathematical meetings''
in the Independent University of Moscow in January~2011.
 I~learned that the question about formality of the cochain
DG\+algebras of absolute Galois groups was asked in the paper~\cite{HW}
from a conversation with I.~Efrat while visiting Ben Gurion University
of the Negev in Be'er Sheva in the Summer of~2014.
 The decision to write this paper was stimulated by
the correspondence with P.~Schneider in the Summer of~2015.
 I~am grateful to all the mentioned persons and institutions.
 Finally, I~would like to thank J.~Stasheff, who read the first
version of this paper, made several helpful comments, and suggested
a number of relevant references.
 I~also wish to thank the anonymous referee for careful reading of
the manuscript and helpful suggestions.
 The author's research is supported by the Israel Science Foundation
grant~\#\,446/15 at the University of Haifa.

\Section{Koszulity Implies $K(\pi,1)$-ness}
\label{koszul-implies-kpi1}

 Postponing the discussion of DG\+algebras, DG\+coalgebras, and
derived nonhomogeneous Koszul duality to
Sections~\ref{koszul-implies-quasiform}\+-%
\ref{kpi1+quasiform-imply-koszul}, we devote this section to
the formulation of a categorical version of the ``Koszulity
implies $K(\pi,1)$'' claim.
 We begin our discussion with recalling some basic definitions and
results from~\cite{PV} and~\cite[Section~5]{Pbogom}.

 A coassociative counital coalgebra $D$ over a field~$k$ is said
to be \emph{coaugmented} if it is endowed with a coalgebra morphism
$k\rarrow D$ (called the \emph{coaugmentation}).
 The quotient coalgebra (without counit) of a coaugmented coalgebra
$D$ by the image of the coaugmentation morphism is denoted by
$D_+=D/k$.

 A coaugmented coalgebra $C$ is called \emph{conilpotent} if for
any element $c\in C$ there exists an integer $m\ge 1$ such that
$c$~is annihilated by the iterated comultiplication map
$C\rarrow C_+^{\ot m+1}$.
 (Several references and terminological comments related
to this definition can be found in~\cite[Remark~D.6.1]{Psemi}.)
 The maximal conilpotent subcoalgebra
$\bigcup_m\ker(D\to D_+^{\ot m+1})$ of a coaugmented coalgebra $D$
is denoted by $\Nilp D\subset D$.

 The cohomology algebra of a coaugmented coalgebra $D$ is defined
as the Ext algebra $H^*(D)=\Ext_D^*(k,k)$, where the field~$k$ is
endowed with a left $D$\+comodule structure via the coaugmentation map.
 The cohomology algebra $H^*(D)$ is computed by the reduced cochain
DG\+algebra of the coalgebra $D$
$$
 k\lrarrow D_+\lrarrow D_+\ot_k D_+\lrarrow D_+\ot_k D_+\ot_k D_+
 \lrarrow\dotsb,
$$
which is otherwise known as the \emph{reduced cobar-complex} or
the \emph{cobar construction} of the coaugmented coalgebra $D$
and denoted by $\Cob^\bu(D)$.

 The following result can be found in~\cite[Corollary~5.3]{Pbogom}.

\begin{thm} \label{maximal-conilpotent-subcoalgebra-cohomology}
 Let $D$ be a coaugmented coalgebra over a field~$k$ and\/
$\Nilp D\subset D$ be its maximal conilpotent subcoalgebra.
 Assume that the cohomology algebra $H^*(D)=\Ext_D^*(k,k)$ is Koszul.
 Then the embedding\/ $\Nilp D\rarrow D$ induces a cohomology
isomorphism $H^*(\Nilp D)\simeq H^*(D)$.
\end{thm}

 Theorem~\ref{maximal-conilpotent-subcoalgebra-cohomology} has
a version with an augmented algebra $R$ replacing the coaugmented
coalgebra~$D$ \cite[Remark~5.6]{Pbogom}.
 Let $R_+=\ker(R\to k)$ denote the augmentation ideal, and let $I$
run over all the ideals $I\subset R_+$ in $R$ for which the quotient
algebra $R/I$ is finite-dimensional and its augmentation ideal $R_+/I$
is nilpotent.
 The coalgebra of pronilpotent completion $R\comp$ of the augmented
algebra $R$ is defined as the filtered inductive limit $R\comp=
\varinjlim_I (R/I)^*$ of the coalgebras $(R/I)^*$ dual to
the finite-dimensional algebras $R/I$.
 Clearly, the coalgebra $R\comp$ is conilpotent.

 The cohomology algebra $H^*(R)=\Ext_R^*(k,k)$ of an augmented algebra
$R$ is computed by its reduced cobar-complex $\Cob^\bu(R)$
$$
 k\lrarrow R_+^*\lrarrow (R_+\ot_k R_+)^*\lrarrow
(R_+\ot_k R_+\ot_k R_+)^*\lrarrow\dotsb
$$
 The natural injective morphism of cobar-complexes $\Cob^\bu(R\comp)
\rarrow\Cob^\bu(R)$ induces a natural morphism of cohomology algebras
$H^*(R\comp)\rarrow H^*(R)$.

\begin{thm} \label{pronilpotent-completion-coalgebra-cohomology}
 Let $R$ be an augmented algebra over a field~$k$ and $R\comp$
be the coalgebra of its pronilpotent completion.
 Assume that the cohomology algebra $H^*(R)=\Ext_R^*(k,k)$ is Koszul.
 Then the natural morphism of the cohomology algebras $H^*(R\comp)
\rarrow H^*(R)$ is an isomorphism.
\end{thm}

 The proofs of
Theorems~\ref{maximal-conilpotent-subcoalgebra-cohomology}
and~\ref{pronilpotent-completion-coalgebra-cohomology}
are based on the following result about the cohomology of conilpotent
coalgebras~\cite[Main Theorem~3.2]{PV}.
 For any positively graded algebra $H^*$ over a field~$k$, we denote by
$\q H^*$ the quadratic part of the algebra $H^*$, i.~e., the universal
final object in the category of quadratic algebras over~$k$
endowed with a morphism into~$H^*$.
 The quadratic algebra $\q H^*$ is uniquely defined by the condition
that the morphism of graded algebras $\q H^*\rarrow H^*$ is
an isomorphism in degree~$1$ and a monomorphism in degree~$2$.

\begin{thm} \label{with-vishik-main-theorem}
 Let $C$ be a conilpotent coaugmented coalgebra, i.~e., $\Nilp C=C$.
 Assume that
\begin{itemize}
\item the quadratic part\/ $\q H^*(C)$ of the graded algebra $H^*(C)$
      is Koszul; and
\item the morphism of graded algebras\/ $\q H^*(C)\rarrow H^*(C)$ is
      an isomorphism in degree~$2$ and a monomorphism in degree~$3$.
\end{itemize}
 Then the graded algebra $H^*(C)$ is quadratic (and consequently,
Koszul). \qed
\end{thm}

 The proof of Theorem~\ref{maximal-conilpotent-subcoalgebra-cohomology}
can be found in~\cite[Theorem~5.2 and Corollary~5.3]{Pbogom}.
 The proof of Theorem~\ref{pronilpotent-completion-coalgebra-cohomology}
is very similar; let us briefly explain how it works.

\begin{proof}[Proof of
Theorem~\ref{pronilpotent-completion-coalgebra-cohomology}]
 One notices that for any augmented algebra $R$ the morphism of
cohomology algebras $H^*(R\comp)\rarrow H^*(R)$ is an isomorphism
in degree~$1$ and a monomorphism in degree~$2$.
 Indeed, the category of left comodules over $R\comp$ is
isomorphic to the full subcategory in the category of left
$R$\+modules consisting of all the ind-nilpotent $R$\+modules
(direct limits of iterated extensions of the trivial $R$\+module~$k$,
the latter being defined in terms of the augmentation of~$R$).
 This is a full subcategory closed under subobjects, quotient
objects, and extensions in the abelian category of left $R$\+modules;
so the argument of~\cite[Lemma~5.1]{Pbogom} applies.

 Now if the algebra $H^*(R)$ is Koszul, then it follows that
the maps $H^1(R\comp)\rarrow H^1(R)$ and $H^2(R\comp)\rarrow
H^2(R)$ are isomorphisms, the composition $\q H^*(R\comp)\rarrow
H^*(R\comp)\rarrow H^*(R)$ is an isomorphism of graded algebras,
and the algebra $H^*(R\comp)$ satisfies the conditions of
Theorem~\ref{with-vishik-main-theorem}.
 Hence we conclude that the algebra $H^*(R\comp)$ is quadratic
and the morphism $H^*(R\comp)\rarrow H^*(R)$ is an isomorphism.
\end{proof}

 A generalization of the results of
Theorems~\ref{maximal-conilpotent-subcoalgebra-cohomology}
and~\ref{pronilpotent-completion-coalgebra-cohomology}
to t\+structures in triangulated categories~\cite[n$^\circ$\,1.3]{BBD}
was announced in~\cite[Remark~5.6]{Pbogom}.
 The idea of this generalization can be described as follows.

 Recall that for any t\+structure $(\sD^{\le0},\sD^{\ge0})$ on
a triangulated category $\sD$ with the core
$\sC=\sD^{\le0}\cap\sD^{\ge0}$ and for any two objects $X$, $Y\in\sC$
there are natural maps
$$
 \theta^n_{\sC,\sD}(X,Y)\:\Ext_\sC^n(X,Y)\lrarrow\Hom_\sD(X,Y[n]),
 \qquad n\ge0,
$$
from the Ext groups in the abelian category $\sC$ to the Hom groups
in the triangulated category~$\sD$.
 The maps $\theta^n_{\sC,\sD}=\theta^n_{\sC,\sD}(X,Y)$ transform
the compositions of Yoneda Ext classes in $\sC$ into the compositions
of morphisms in~$\sD$.
 Furthermore, the maps $\theta^n_{\sC,\sD}$ are always isomorphisms for
$n\le1$ and monomorphisms for $n=2$ (see~\cite[Remarque~3.1.17]{BBD},
\cite[Section~4.0]{BGSch}, or~\cite[Corollary~A.17]{Partin}).

 Starting with a coaugmented coalgebra $D$, consider the conilpotent
coalgebra $C=\Nilp D$ and the abelian category $\sC$ of
finite-dimensional left $C$\+comodules.
 Consider the bounded derived category of left $D$\+comodules
$\sD^\b(D\comodl)$, and set $\sD$ to be the full subcategory of
$\sD^\b(D\comodl)$ generated by the abelian subcategory
$\sC\subset D\comodl$.
 Then $\sC$ is the core of a bounded t\+structure on~$\sD$.

 Analogously, starting with an augmented algebra $R$, consider
the conilpotent coalgebra $C=R\comp$ and the abelian category $\sC$
of finite-dimensional left $C$\+comodules (or, which is the same,
finite-dimensional nilpotent $R$\+modules).
 Consider the bounded derived category of left $R$\+modules
$\sD^\b(R\modl)$, and set $\sD$ to be the full triangulated subcategory
of $\sD^\b(R\modl)$ generated by the abelian subcategory
$\sC\subset R\modl$.
 Once again, $\sC$ is the core of a bounded t\+structure on~$\sD$.

 In both cases, the assertions of
Theorems~\ref{maximal-conilpotent-subcoalgebra-cohomology}
and~\ref{pronilpotent-completion-coalgebra-cohomology} claim that
the map
$$
 \theta^n_{\sC,\sD}(k,k)\:\Ext_\sC^n(k,k)\rarrow\Hom_\sD(k,k[n])
$$
is an isomorphism for all~$n$, provided that the graded algebra
$\Hom_\sD(k,k[*])$ is Koszul.
 Here the trivial $D$\+comodule or $R$\+module~$k$ is the only
irreducible object in~$\sC$.
 All objects of the abelian category $\sC$ being of finite length,
it follows that all the morphisms $\theta^n_{\sC,\sD}$ are isomorphisms
for the t\+structures under consideration.

 A t\+structure for which all the maps $\theta^n_{\sC,\sD}$ are
isomorphisms is called a ``t\+structure of derived
type''~\cite[Section~4.0]{BGSch}.
 This condition is also known as the ``$K(\pi,1)$ condition of
Bloch and Kriz''~\cite{BK} and, in somewhat larger generality, as
the ``silly filtration condition'' \cite[Sections~0.2\+-0.5]{Partin}.
 Proving that a given t\+structure is of derived type is sometimes
an important and difficult problem (see, e.~g., \cite{Beil}).
 A standard approach working in some particular cases can be found
in~\cite[Section~12]{Kel} (see also~\cite[Section~A.2]{Pcosh});
the results below in this section provide an alternative way.

\bigskip

 Let $S=\{\alpha\}$ be a set of indices.
 A \emph{big ring} (or a ``ring with many objects'') $A$ is
a collection of abelian groups $A^n_{\alpha\beta}$ endowed with
the multiplication maps $A_{\alpha\beta}\times A_{\beta\gamma}
\rarrow A_{\alpha\gamma}$ and the unit elements $1_{\alpha}\in
A_{\alpha\alpha}$ satisfying the conventional associativity
and unit axioms.
 A big ring with a set of indices $S$ is the same thing as
a preadditive category with the objects indexed by~$S$
(see~\cite{Mit} or~\cite[Section~A.1]{Partin}).

 The categories of left and right modules over a big ring are
defined in the obvious way.
 A left (resp., right) $A$\+module is the same thing as a covariant
(resp., contravariant) additive functor from the preadditive
category corresponding to $A$ to the category of abelian groups.
 The \emph{tensor product} of a right $A$\+module $N$ and a left
$A$\+module $M$ is an abelian group constructed as the cokernel of
the difference of the right and left action maps
$$
 \bigoplus\nolimits_{\alpha,\beta} N_\alpha\ot_\Z A_{\alpha,\beta}\ot_\Z
 M_\beta\,\rightrightarrows\,
 \bigoplus\nolimits_\alpha N_\alpha\ot_\Z M_\alpha.
$$
 The derived functor $\Tor^A$ of the functor of tensor product~$\ot_A$
of modules over a big graded ring is defined in the usual way.

 A \emph{big graded ring} (or a ``graded ring with many objects'')
$A$ is a big ring in which every group $A_{\alpha\beta}$ is graded and
the multiplication maps are homogeneous.
 In other words, it is a collection of abelian groups $A^n_{\alpha\beta}$
endowed with the multiplication maps $A^p_{\alpha\beta}\times
A^q_{\beta\gamma}\rarrow A^{p+q}_{\alpha\gamma}$ and the unit
elements $1_\alpha\in A^0_{\alpha\alpha}$ satisfying
the associativity and unit axioms.
 We will assume that $A^n_{\alpha\beta}=0$ for $n<0$ or $\alpha\ne\beta$
and $n=0$, and that the rings $A^0_{\alpha\alpha}$ are (classically) 
semisimple.

 The definition of the Koszul property of a nonnegatively graded ring
$A=A_0\oplus A_1\oplus A_2\oplus\dotsb$ with a semisimple degree-zero
component $A_0$ is pretty well known~\cite{BGSoe}, and the definition
of a Koszul big graded ring in the above generality is its
straightforward extension.
 Specifically, a big graded ring $A$ is called \emph{Koszul} if one
has $\Tor^A_{ij}(A^0_{\alpha\alpha},A^0_{\beta\beta})=0$ for all $\alpha$,
$\beta\in S$ and all $i\ne j$ (where $i$~denotes the homological
and $j$~the internal grading on the $\Tor$).
 One can also define quadratic big graded rings and the quadratic
part of a big graded ring, etc.
 A discussion of the Koszul property of big graded rings in
a greater generality with the semisimplicity condition replaced by
a flatness condition can be found in~\cite[Section~7.4]{Partin},
and even without the flatness condition, in the rest
of~\cite[Section~7]{Partin}.

  The following theorem is the main result of this section.

\begin{thm} \label{t-structure-koszulity-main-theorem}
 Let\/ $\sC$ be the core of a t\+structure on a small triangulated
category\/~$\sD$.
 Suppose that every object of\/ $\sC$ has finite length and
let $I_\alpha$ be the irreducible objects of\/~$\sC$.
 Assume that the big graded ring $A$ with the components
$A^n_{\alpha,\beta}=\Hom_\sD(I_\alpha,I_\beta[n])$ is Koszul.
 Then for any two objects $X$, $Y\in\sC$ and any $n\ge0$ the natural
map $\theta^n_{\sC,\sD}(X,Y)\:\Ext^n_\sC(X,Y)\rarrow\Hom_\sD(X,Y[n])$
is an isomorphism.
\end{thm}

 The above theorem is deduced from the following result about
the Ext rings between irreducible objects in abelian categories,
which is a categorical generalization of
Theorem~\ref{with-vishik-main-theorem}.

\begin{thm} \label{ext-in-abelian-category-main-theorem}
 Let\/ $\sC$ be a small abelian category such that every object of\/
$\sC$ has finite length and  let $I_\alpha$ be the irreducible objects
of\/~$\sC$.
 Consider the big graded ring $B$ with the components
$B^n_{\alpha,\beta}=\Ext^n_\sC(I_\alpha,I_\beta)$.
 Assume that
\begin{enumerate}
\renewcommand{\theenumi}{\Roman{enumi}}
  \item the quadratic part\/ $\q B$ of the big graded ring $B$
        is Koszul; and
  \item the morphism of big graded rings\/ $\q B\rarrow B$ is
        an isomorphism in the degree~$n=2$ and a monomorphism in
        the degree~$n=3$.
\end{enumerate}
 Then the big graded ring $B$ is quadratic (and consequently, Koszul).
\end{thm}

 Actually, the following slightly stronger form of
Theorem~\ref{t-structure-koszulity-main-theorem}, generalizing both
Theorems~\ref{t-structure-koszulity-main-theorem}
and~\ref{ext-in-abelian-category-main-theorem}, can be obtained from
Theorem~\ref{ext-in-abelian-category-main-theorem}.
 It is a categorical generalization of~\cite[Theorem~5.2]{Pbogom}.

\begin{thm} \label{t-structure-koszulity-stronger-version}
 Let\/ $\sC$ be the core of a t\+structure on a small triangulated
category\/~$\sD$.
 Suppose that every object of\/ $\sC$ has finite length and
let $I_\alpha$ be the irreducible objects of\/~$\sC$.
 Consider the big graded rings $A$ and $B$ with the components
$A^n_{\alpha,\beta}=\Ext^n_\sC(I_\alpha,I_\beta)$
and $B^n_{\alpha,\beta}=\Hom_\sD(I_\alpha,I_\beta[n])$.
 Then whenever the big graded ring $B$ satisfies the assumptions
(I) and~(II) of Theorem~\ref{ext-in-abelian-category-main-theorem},
the natural morphism of big graded rings $A\rarrow B$ induces
an isomorphism $A\simeq\q B$.
\end{thm}

 Let us clarify the following point, which otherwise might become
a source of confusion.
 It is well-known~\cite{BBD,BGSch,Partin} that for a t-structure with
the core $\sC$ on a triangulated category $\sD$ the natural
maps $\theta^n_{\sC,\sD}(X,Y)\:\Ext^n_\sC(X,Y)\rarrow\Hom_\sD(X,Y[n])$
are isomorphisms for all $X$, $Y\in\sC$, \,$n\ge0$ if and only if
any element in $\Hom_\sD(X,Y[n])$ can be decomposed into
a product of $n$~elements from the groups $\Hom_\sD(U,V[1])$
with $U$, $V\in\sC$.
 However, this is only true because one considers
the degree-one generation condition for $X$ and $Y$ running over
\emph{all} the objects of $\sC$ and not just the irreducible objects
(cf.~\cite[Proposition~B.1]{Partin}).
 The Koszulity condition in
Theorem~\ref{t-structure-koszulity-main-theorem} is much stronger
than a degree-one generation condition, but it is applied
to a much smaller algebra of homomorphisms between the irreducible
objects of~$\sC$.
 On the other hand, the big graded ring of Yoneda extensions between
all the objects of a given abelian category is always Koszul in
an appropriate sense~\cite[Example~8.3]{Partin}.

\begin{ex}
 Given a discrete group $\Gamma$ and a field~$k$, set $C$ to be
the coalgebra of (functions on) the proalgebraic completion of
$\Gamma$ over~$k$.
 Then the category $\sC$ of finite-dimensional representations
of a group $\Gamma$ over a field~$k$ is isomorphic to the category of
finite-dimensional left $C$\+comodules.
 Let $\sD$ denote the full triangulated subcategory of the bounded
derived category $\sD^\b(k[\Gamma]\modl)$ of arbitrary $\Gamma$\+modules
over~$k$ generated by the subcategory of finite-dimensional modules
$\sC\subset k[\Gamma]\modl$.
 Then $\sC$ is the core of a (bounded) t\+structure on $\sD$, so
Theorems~\ref{t-structure-koszulity-main-theorem}--%
\ref{t-structure-koszulity-stronger-version} are applicable whenever
the Koszulity assumption or the assumptions (I\+-II) are satisfied.
 This would allow to obtain a comparison between the cohomology
of (the proalgebraic group corresponding to) the coalgebra/commutative
Hopf algebra $C$ and the discrete group $\Gamma$ with
finite-dimensional coefficients.
\end{ex}

 Let us emphasize that we do \emph{not} assume the abelian category
$\sC$ or the triangulated category $\sD$ to be linear over a field in
the above theorems.
 E.~g., the category of finite modules over a profinite group is fine
in the role of $\sC$, as is the category of modules of finite
length over any complete commutative local ring, etc.

\begin{proof}[Proof of
Theorem~\ref{t-structure-koszulity-stronger-version}]
 This is a particular case of~\cite[Corollary~8.5]{Partin}.
 By a general property of t\+structures
(see~\cite[Remarque~3.1.17]{BBD} or~\cite[Corollary~A.17]{Partin}),
the morphism of big graded rings $A\rarrow B$ is an isomorphism
in degree~$1$ and a monomorphism in degree~$2$ (cf.\ the proof
of Theorem~\ref{pronilpotent-completion-coalgebra-cohomology}).
 It follows that if the ring $B$ satisfies the conditions~(I)
and~(II), then so does the ring~$A$.
 By Theorem~\ref{ext-in-abelian-category-main-theorem}, one can
then conclude that the big graded ring $A$ is quadratic, and
hence $A\simeq \q B$.
\end{proof}

\begin{proof}[Proof of
Theorem~\ref{t-structure-koszulity-main-theorem}]
 According to Theorem~\ref{t-structure-koszulity-stronger-version},
the map $\theta_{\sC,\sD}^n(X,Y)$ is an isomorphism whenever
the objects $X$ and $Y\in\sC$ are irreducible.
 The general case follows by induction on the lengths.
\end{proof}

\begin{proof}[Brief sketch of proof of
Theorem~\ref{ext-in-abelian-category-main-theorem}]
 This is a particular case of~\cite[Theorem~8.4]{Partin}.

 Let us start with a comment on the proof of
Theorem~\ref{with-vishik-main-theorem}.
 It is based on two ingredients: the basic theory of quadratic and
Koszul graded algebras and coalgebras over a field, and the spectral
sequence connecting the cohomology of a conilpotent coalgebra $C$ with
the cohomology of its associated graded coalgebra with respect to
the coaugmentation filtration.
 In the case of Theorem~\ref{ext-in-abelian-category-main-theorem},
first of all one needs to develop the basic theory of quadratic and
Koszul big graded rings, which is done (in a greater generality)
in~\cite[Sections~6\+-7]{Partin}.
 Then one has to work out the passage from the ungraded category $\sC$
to its graded version.

 One possible approach to the latter task would be to associate
a coalgebra-like algebraic structure with an abelian category consisting
of objects of finite length, then filter that structure and pass to
the associated graded one.
 The required class of algebraic structures was introduced
in~\cite[\S\,IV.3\+4]{Gab}.
 An abelian category consisting of objects of finite length is
equivalent to the category of finitely generated discrete modules over
a \emph{pseudo-compact} topological ring.
 One would have to embed the category of finitely generated discrete
modules into the category of pseudo-compact modules in order to do
cohomology computations with projective resolutions.

 There is a more delicate approach developed
in~\cite[Sections~3\+-4]{Partin}, which is purely categorical.
  One associates with an abelian category $\sC$ consisting of objects
of finite length the exact category $\sF$ whose objects are
the objects of $\sC$ endowed with a finite filtration for which all
the successive quotient objects are semisimple.
 Then one needs to pass from the filtered exact category $\sF$ to
the associated graded abelian category~$\sG$.
 The main property connecting the categories $\sC$ and $\sF$ is
the natural isomorphism
\begin{equation}  \label{filtered-inductive-limit-isomorphism}
 \Ext^n_\sC(u(X),u(Y))\.\simeq\.
 \varinjlim\nolimits_{m\to+\infty}\Ext^n_\sF(X,Y(m)),
\end{equation}
where $u\:\sF\rarrow\sC$ is the functor of forgetting the filtration
and $Z\longmapsto Z(m)$ denotes the filtration shift.
 The main property connecting the categories $\sF$ and $\sG$ is
the long exact sequence
\begin{multline} \label{associated-graded-bockstein-sequence}
 \dotsb\lrarrow\Ext^n_\sF(X,Y(-1))\lrarrow\Ext^n_\sF(X,Y) \\ \lrarrow
 \Ext^n_\sG(\gr X\;\gr Y)\lrarrow\Ext^{n+1}_\sF(X,Y(-1))\lrarrow\dotsb
\end{multline}
for any two objects $X$, $Y\in\sF$, where $\gr\:\sF\rarrow\sG$
is the functor assigning to a filtered object its associated
graded object.
 The construction of the category $\sG$ is a particular case of
a general construction of the reduction of an exact category by
a graded center element developed in the paper~\cite{Pred}.

 The isomorphism~\eqref{filtered-inductive-limit-isomorphism} and
the long exact sequence~\eqref{associated-graded-bockstein-sequence}
taken together are used in lieu of the spectral sequence connecting
the cohomology of a coalgebra $C$ and its associated graded coalgebra
$\gr_NC$ by the coaugmentation filtration $N$ in order to extend
the argument from~\cite[Main Theorem~3.2]{PV} from conilpotent
coalgebras to abelian categories consisting of objects of
finite length.
\end{proof}

\Section{Koszulity Implies Quasi-Formality}
\label{koszul-implies-quasiform}

 Generally speaking, \emph{Massey products} are natural partially
defined multivalued polylinear operations in the cohomology algebra
of a DG\+algebra which are preserved by quasi-isomorphisms of
DG\+algebras.
 There are several different ways to construct such operations.
 We start with introducing the construction most relevant for our
purposes, and later explain how it is related to a more elementary
construction.
 The most relevant reference for us is~\cite{May1}; see also
the earlier paper~\cite{Sta}, the heavier~\cite{May2}, and
the later work~\cite[Section~5]{GM}.

 Let $A^\bu=(A^*\;d\:A^i\to A^{i+1})$ be a nonzero DG\+algebra over
a field~$k$; suppose that it is endowed with an augmentation
(DG\+algebra morphism) $A^\bu\rarrow k$ and denote the augmentation
kernel ideal by $A_+^\bu=\ker(A^\bu\to k)$.
 By the definition, the \emph{bar construction} $\Br^\bu(A^\bu)$ of
an augmented DG\+algebra $A^\bu$ is the tensor coalgebra
$$\textstyle
 \Br(A)=\bigoplus_{n=0}^\infty A_+[1]^{\ot n}
$$
of the graded vector space $A_+[1]$ obtained by shifting by~$1$
the cohomological grading of the augmentation ideal~$A_+$.
 The grading on $\Br(A)$ is induced by the grading of $A_+[1]$.
 Alternatively, one can define the bar construction $\Br(A)$ as
the direct sum of the tensor powers $A_+^{\ot n}$ of the vector space
$A_+$ and endow it with the total grading equal to the difference $i-n$
of the grading~$i$ induced by the grading of $A_+$ and the grading~$n$
by the number of tensor factors.

 The differential on $\Br^\bu(A^\bu)$ is the sum of two summands $d+\d$,
the former of them induced by the differential on $A_+^\bu$ and
the latter one by the multiplication in~$A_+^\bu$.
 One has to work out the plus/minus signs in order to make the total
differential on $\Br^\bu(A^\bu)$ square to zero; this is
a standard exercise.

 One defines a natural increasing filtration on the complex
$\Br^\bu(A^\bu)$ by the rule
$$\textstyle
 F_p\Br^\bu(A^\bu)=\bigoplus_{n=0}^p A_+[1]^{\ot n}. 
$$
 The associated graded complex $\gr^F\Br^\bu(A^\bu)$ of the complex
$\Br^\bu(A^\bu)$ by the filtration $F$ is naturally identified with
the graded vector space $\Br(A)$ endowed with the differential~$d$
induced by the differential on~$A_+^\bu$.
 The spectral sequence $E_r^{pq}$ of the filtered complex
$\Br^\bu(A^\bu)$ has the initial page
$$
 E_1^{pq}=(H^*(A_+^\bu)^{\ot p})^q,
$$
where the grading~$q$ on the tensor power $H^*(A_+^\bu)^{\ot p}$ of
the augmentation ideal $H^*(A_+^\bu)$ of the cohomology algebra
$H^*(A^\bu)$ of the DG\+algebra $A$ is induced by the cohomological
grading on $H^*(A_+^\bu)$.
 The differentials are
$$
 d_r^{pq}\:E_r^{pq}\lrarrow E_r^{p-r,q-r+1},
$$
and the limit page is given by the rule
$$
 E_\infty^{pq}=\gr^F_pH^{q-p}\Br^\bu(A^\bu).
$$

 The cohomology of the complex $\Br^\bu(A^\bu)$ are known as
the \emph{differential Tor spaces}~\cite{EM1,GM}
$H^*\Br^\bu(A^\bu)=\Tor^{A^\bu}_*(k,k)$ (``of the first
kind''~\cite{HMS}) over the DG\+alge\-bra~$A^\bu$.
 This is the derived functor of tensor product of DG\+modules
over $A^\bu$ defined on the ``conventional'' derived categories
of DG\+modules (obtained by inverting the DG\+module
morphisms inducing isomorphisms of the cohomology modules);
see~\cite{Kel0}, \cite{Hin1}, or~\cite[Section~1]{Pkoszul}.
 The spectral sequence $E_r^{pq}$ is called the \emph{algebraic
Eilenberg--Moore spectral sequence} associated with
a DG\+algebra~$A^\bu$ \cite{EM1,GM,HS}.

 The differential~$d_1$ is induced by the multiplication in
the cohomology algebra $H^*(A^\bu)$, and whole the page $E_1$ is
simply the bar-complex of the cohomology algebra $H^*(A^\bu)$.
 Hence the page $E_2$ can be computed as
$$
 E_2^{pq}=\Tor^{H^*(A^\subbu)}_{pq}(k,k),
$$
where the first grading~$p$ on the Tor spaces in the right-hand
side is the conventional homological grading of the Tor and
the second grading~$q$ is the ``internal'' grading induced by
the cohomological grading on the algebra $H^*(A^\bu)$.

 The differentials $d_r^{pq}$, \,$r\ge2$, in the algebraic
Eilenberg--Moore spectral sequence $E_r^{pq}=E_r^{pq}(A^\bu)$ associated
with an augmented DG\+algebra $A^\bu$ are, by the definition,
the \emph{Massey products} in the cohomology  algebra $H^*(A^\bu)$
that we are interested in.
 An augmented DG\+algebra $A^\bu$ is called \emph{quasi-formal} if
the spectral sequence $E_r^{pq}(A^\bu)$ degenerates at the page~$E_2$,
that is all the Massey products $d_r^{pq}$, \,$r\ge2$, vanish.

 An augmented DG\+algebra $A^\bu$ is called \emph{formal} (in
the class of augmented DG\+algebras) if it can be connected with
its cohomology algebra $H^*(A^\bu)$, viewed as an augmented
DG\+algebra with zero differential and the augmentation induced
by that of $A^\bu$, by a chain of quasi-isomorphisms of augmented
DG\+algebras.

\begin{prop}
 Any formal augmented DG\+algebra is quasi-formal.
\end{prop}

\begin{proof}
 Notice that any morphism of augmented DG\+algebras $A^\bu\rarrow B^\bu$
induces a morphism of spectral sequences $E_r^{pq}(A^\bu)\rarrow
E_r^{pq}(B^\bu)$.
 When the morphism $A^\bu\rarrow B^\bu$ is a quasi-isomorphism,
the induced morphism of spectral sequences is an isomorphism on
the pages $E_1$, and consequently also on all the higher pages.
 So the Massey products are preserved by quasi-isomorphisms of
augmented DG\+algebras.

 Therefore, whenever an augmented DG\+algebra $A^\bu$ is connected with
an augmented DG\+algebra $B^\bu$ by a chain of quasi-isomorphisms of
augmented DG\+algebras, an augmented DG\+algebra $A^\bu$ is quasi-formal
if and only if an augmented DG\+algebra $B^\bu$ is.
 Since the Massey products in a DG\+algebra with zero differential
clearly vanish (as do the higher differentials in the spectral sequence
of any bicomplex with one of the two differentials vanishing),
the desired assertion follows.
\end{proof}

 An extension of (a part of) the canonical, partially defined,
multivalued Massey operations on the cohomology algebra of
a DG\+algebra $A^\bu$ to total, single-valued polylinear maps,
which taken together are defined up to certain transformations,
is called the \emph{$\mathrm A_\infty$\+algebra} structure on
the cohomology algebra $H^*(A^\bu)$ of a DG\+algebra~$A^\bu$
\cite{Sta,Kad0,Kad,Gug,LPWZ}.
 The $\mathrm A_\infty$\+algebra structure on $H^*(A^\bu)$ determines
a DG\+algebra $A^\bu$ up to quasi-isomorphism.
 Thus, while vanishing of the Massey operations in the cohomology
only makes a DG\+algebra $A^\bu$ \emph{quasi-formal}, vanishing of
the higher operations in (a certain representative of
the $\mathrm A_\infty$\+isomorphism class of)
the $\mathrm A_\infty$\+algebra structure on $H^*(A^\bu)$ would
actually mean that the DG\+algebra $A^\bu$ is \emph{formal}.

\bigskip

 Let $m\ne0$ be an integer.
 Suppose that the cohomology algebra $H^*(A^\bu)$ is concentrated in
the cohomological gradings $q=mn$, \,$n=0$,~$1$, $2$,~\dots, one has
$H^0(A^\bu)=k$, and the algebra $H^*(A^\bu)$ is Koszul in the grading
rescaled by~$m$, i.~e., one has
$$
 \Tor^{H^*(A^\subbu)}_{pq}(k,k)=0  \qquad\text{for \,$mp\ne q$.}
$$
 Then all the differentials $d_r^{pq}$, \,$r\ge2$, vanish for
``dimension'' (bigrading) reasons, and the DG\+algebra $A^\bu$ is
quasi-formal (cf.~\cite{Ber}).
 One can say that this is an instance of \emph{intrinsic
quasi-formality}, i.~e., a situation when any augmented DG\+algebra
with a given cohomology algebra is quasi-formal.
 In this paper, we are interested in the case $m=1$, i.~e.,
the situation when the augmentation ideal $H^*(A^\bu_+)$ of
the cohomology algebra $H^*(A^\bu)$ is concentrated in
the cohomological degrees~$1$, $2$, $3$,~\dots

\begin{cor}
 Suppose that the cohomology algebra $H^*(A^\bu)$ of an augmented
DG\+algebra $A^\bu$ is positively cohomologically graded and Koszul in
its cohomological grading.
 Then the augmented DG\+algebra $A^\bu$ is quasi-formal.
\end{cor}

\begin{proof}
 This is a corollary of the definitions, as explained above.
\end{proof}

 For lack of a better term, let us call the Massey products discussed
above the \emph{tensor Massey products}.
 Our next aim is to compare these with a more elementary construction
that we call the \emph{tuple Massey products}.

 One reason for our interest in tuple Massey products and this
comparison comes from the application to the absolute Galois groups
and Galois cohomology.
 A conjecture of ours claims that the cohomology algebra $H^*(G_F,\Z/l)$
of the absolute Galois group $G_F$ of a field $F$ containing
a primitive $l$\+root of unity is Koszul~\cite{PV,Partin}.
 In the paper~\cite{Pnum}, this conjecture was proven for all
the (one-dimensional) local and global fields.
 On the other hand, there is a series of recent
papers~\cite{HW,MT1,EM,MT2} discussing and partially proving
the conjecture that tuple Massey products of degree-one elements
vanish in the cohomology algebra $H^*(G_F,\Z/l)$.
 The results above in this section and the discussion below show that
the Koszulity conjecture implies vanishing of the tensor Massey
products in $H^*(G_F,\Z/l)$, but may have no direct implications
concerning the problem of vanishing of the tuple Massey products.

\bigskip

 Let $A^\bu=(A^*\;d\:A^i\to A^{i+1})$ be a DG\+algebra over a field~$k$;
assume for simplicity that $A^i=0$ for $i<0$ and $A^0=k$ (so in
particular $d^0\:A^0\rarrow A^1$ is a zero map and the DG\+algebra
$A^\bu$ has a natural augmentation $A^\bu\rarrow k$).
 Let $B^i\subset Z^i\subset A^i$ denote the subspaces of coboundaries
and cocycles in $A^\bu$, so that $H^i=H^i(A^\bu)=Z^i/B^i$.
 The simplest possible construction of a $3$\+tuple Massey product of
degree-one elements in the cohomology algebra $H^*(A^\bu)$ proceeds
as follows.

 Let $x$, $y$, $z\in H^1(A^\bu)$ be three elements for which
$xy=0=yz$ in $H^2(A^\bu)$.
 Pick some preimages $\tilde x$, $\tilde y$, $\tilde z\in Z^1$ of
the elements $x$, $y$, $z\in H^1$.
 Then the products $\tilde x\tilde y$ and $\tilde y\tilde z$ are
coboundaries in $A^2$; so there exist elements $\zeta$ and $\xi\in A^1$
such that $d\zeta=\tilde x\tilde y$ and $d\xi=\tilde y\tilde z$ in~$A^2$.
 Hence one has
$$
 d(\tilde x\xi+\zeta\tilde z) = -\tilde xd(\xi)+d(\zeta)\tilde z
 = -\tilde x\tilde y\tilde z+\tilde x\tilde y\tilde z=0,
$$
so the element $\tilde x\xi+\zeta\tilde z$ is a cocycle in~$A^2$.
 By the definition, one sets the $3$\+tuple Massey product
$\langle x,y,z\rangle\in H^2(A^\bu)$ to be equal to the cohomology
class of the cocycle $\tilde x\xi+\zeta\tilde z\in Z^2$.

 We have made some arbitrary choices along the way, so it is important
to find out how does the output depend on these choices.
 Replacing the cochain~$\zeta$ by a different cochain $\zeta'$ with
the same differential $d\zeta'=\tilde x\tilde y\in A^2$ adds the product
of two cocycles $(\zeta'-\zeta)\tilde z\in Z^1\cdot\tilde z\subset Z^2$
to the cocycle $\tilde x\xi+\zeta\tilde z\in Z^2$.
 This means adding an element of the subspace $H^1\cdot z\subset H^2$
to our $3$\+tuple Massey product $\langle x,y,z\rangle\in H^2(A^\bu)$.

 Similarly, replacing the cochain~$\xi$ by a different cochain
$\xi'$ with the same differential $d\xi'=\tilde y\tilde z$ adds
the product of two cocycles $\tilde x(\xi'-\xi)\in\tilde x\cdot Z^1
\subset Z^2$ to the cocycle $\tilde x\xi+\zeta\tilde z\in Z^2$, which
means adding an element of the subspace $x\cdot H^1\subset H^2$ to
the $3$\+tuple Massey product $\langle x,y,z\rangle$.

 Furthermore, since we have assumed that $A^0=k$, the preimages
$\tilde x$, $\tilde y$, $\tilde z\in Z^1$ of given elements
$x$, $y$, $z\in H^1$ are uniquely defined.
 However, even without the $A^0=k$ assumption, one easily checks that
the choice of the preimages $\tilde x$, $\tilde y$, $\tilde z$
does not introduce any new indeterminacies into the output of our
$3$\+tuple Massey product construction as compared to the ones
we already described.

 To conclude, the tuple Massey product of three elements
$x$, $y$, $z\in H^1(A^\bu)$ with $xy=0=yz$ in $H^2(A^\bu)$ is well-defined
as an element of the quotient space
$$
 \langle x,y,z\rangle \in H^2(A^\bu)/(x\cdot H^1+ H^1\cdot z).
$$

\bigskip

 Now let us describe the connection with the tensor Massey products.
 Suppose that we want to extend the above construction to elements of
the tensor product space $H^1(A^\bu)^{\ot 3} =
H^1(A^\bu)\ot_k H^1(A^\bu)\ot_k H^1(A^\bu)$.
 With any three vectors $x$, $y$, $z\in H^1(A^\bu)$ one can
associate the decomposable tensor $x\ot y\ot z\in H^1(A^\bu)^{\ot 3}$;
however, not every tensor is decomposable.

 Let $m\:A^\bu\ot_k A^\bu\rarrow A^\bu$ denote the multiplication map in
the DG\+algebra $A^\bu$.
 We denote the induced (conventional) multiplication on the cohomology
algebra by $m_2\:H^*(A^\bu)\ot_k H^*(A^\bu)\rarrow H^*(A^\bu)$.
 Let $K^2\subset H^1(A^\bu)\ot_k H^1(A^\bu)$ denote the kernel of
the multiplication map $m_2\:H^1(A^\bu)\ot_k H^1(A^\bu)
\rarrow H^2(A^\bu)$.
 We would like to have our triple Massey product defined
on the subspace
$$
 K^2\ot_kH^1(A^\bu)\cap H^1(A^\bu)\ot_k K^2\.\subset\. H^1(A^\bu)^{\ot 3}.
$$

 The construction proceeds as follows.
 Given a tensor $\theta\in K^2\ot H^1\cap H^1\ot K^2\subset
H^1\ot H^1\ot H^1$, we lift it to a tensor $\widetilde\theta$
in $Z^1\ot Z^1\ot Z^1$ and apply the maps of multiplication of
the first two and the last two tensor factors $m^{(12)}=m\ot\id$ and
$m^{(23)}=\id\ot m$ to obtain a pair of elements 
$m^{(12)}(\widetilde\theta)\in B^2\ot Z^1$ and
$m^{(23)}(\widetilde\theta)\in Z^1\ot B^2$.
 Then we lift these two elements arbitrarily to elements in
$A^1\ot Z^1$ and $Z^1\ot A^1$, respectively, and finally apply
the product map~$m$ to each of them and add the results in order
to obtain an element in~$A^2$.
 By virtue of a computation similar to the above, this turns out
to be an element of~$Z^2$.
 Its image in $H^2(A^\bu)$, denoted by $m_3(\theta)$, is
the triple tensor Massey product of our tensor~$\theta$.

 What is the subspace in $H^2(A^\bu)$ up to which the element
$m_3(\theta)$ is defined?
 Let $W_l\subset H^1$ be the minimal vector subspace for which
$\theta\in W_l\ot H^1\ot H^1$, and let $W_r$ be the similar
minimal subspace for which $\theta\in H^1\ot H^1\ot W_r$
(hence in fact $\theta\in W_l\ot H^1\ot W_r$).
 If one is careful, one can make the Massey product $m_3(\theta)$
well-defined up to elements of $W_l\cdot H^1+H^1\cdot W_r\subset
H^2(A^\bu)$.
 However, generally speaking, for ``most'' tensors $\theta\in
K^2\ot H^1\cap H^1\ot K^2$ (and certainly for ``most'' tensors in
$H^1\ot H^1\ot H^1$) one would expect $W_l=H^1=W_r$.
 So the triple Massey product that we have constructed is most
simply viewed as a linear map
\begin{gather*}
 m_3\:K^2\ot_k H^1(A^\bu)\cap H^1(A^\bu)\ot_k K^2 \lrarrow
 H^2(A^\bu)/m_2(H^1(A^\bu){\ot_k}H^1(A^\bu)), \\
 K^2 = \ker(m_2\:H^1(A^\bu)\ot H^1(A^\bu)\rarrow H^2(A^\bu)).
\end{gather*}

 Notice that one has
$$
 K^2\ot_k H^1(A^\bu)\cap H^1(A^\bu)\ot_k K^2 \.\simeq\.
 \Tor^{H^*(A^\subbu)}_{3,3}(k,k) \.=\. E_2^{3,3}
$$
and
$$
 H^2(A^\bu)/m_2(H^1(A^\bu){\ot_k}H^1(A^\bu)) \.\simeq\.
 \Tor^{H^*(A^\subbu)}_{1,2}(k,k) \.=\. E_2^{1,2}
$$
in the Eilenberg--Moore spectral sequence.
 We have obtained an explicit construction of the differential
$$
 d_2^{3,3}\:E_2^{3,3}\lrarrow E_2^{1,2},
$$
which is the simplest example of a tensor Massey product in
the sense of our definition.

 How is this triple tensor Massey product construction related to
the $3$\+tuple Massey product defined above?
 On the one hand, a subspace $K^2\subset H^1\ot H^1$ may well contain
no nonzero decomposable tensors at all, while containing many
nontrivial indecomposable tensors.
 Then there may be also many nontrivial indecomposable tensors in
$K^2\ot H^1\cap H^1\ot K^2$.
 So the domain of definition of the tensor Massey product may be
essentially much wider than that of the tuple Massey product.
 On the other hand, the latter, more elementary construction may
produce its outputs with better precision, i.~e., modulo a smaller
subspace in $H^2(A^\bu)$.
 Thus the map~$m_3$ carries both more and less information about
the DG\+algebra $A^\bu$ than the operation $\langle x,y,z\rangle$
in the cohomology algebra $H^*(A^\bu)$.

\bigskip

 Similarly, let $x_1$, $x_2$, $x_3$, $x_4\in H^1(A^\bu)$ be four
elements for which $x_1x_2=x_2x_3=x_3x_4=0$ in $H^2(A^\bu)$.
 Since we have assumed that $A^0=k$, these elements have uniquely
defined preimages in $Z^1\simeq H^1$, which we will denote by
$\tilde x_1$, $\tilde x_2$, $\tilde x_3$, $\tilde x_4$.
 The products $\tilde x_1\tilde x_2$, \,$\tilde x_2\tilde x_3$, and
$\tilde x_3\tilde x_4$ are coboundaries in $A^2$, so there exist three
elements $\eta_{12}$, $\eta_{23}$, $\eta_{34}\in A^1$ such that
$d\eta_{rs} = \tilde x_r\tilde x_s$ for all $1\le r<s\le4$, \,$s-r=1$.
 The elements
$$
 \tilde x_1\eta_{23}+\eta_{12}\tilde x_3 \quad\text{and}\quad
 \tilde x_2\eta_{34}+\eta_{23}\tilde x_4 \in Z^2
$$
represent the $3$\+tuple Massey products $\langle x_1,x_2,x_3\rangle$
and $\langle x_2,x_3,x_4\rangle$.
 Suppose that these two cocycles are coboundaries, i.~e., there
exist two elements $\zeta_{123}$ and $\zeta_{234}\in A^1$ such that
$d\zeta_{rst}=\tilde x_r\eta_{st}+\eta_{rs}\tilde x_t$ for
$1\le r<s<t\le4$, \ $t-s=s-r=1$.
 One has
\begin{multline*}
 d(\tilde x_1\zeta_{234}+\eta_{12}\eta_{34}+\zeta_{123}\tilde x_4) \\
= -\tilde x_1\tilde x_2\eta_{34}-\tilde x_1\eta_{23}\tilde x_4 +
\tilde x_1\tilde x_2\eta_{34}-\eta_{12}\tilde x_3\tilde x_4 +
\tilde x_1\eta_{23}\tilde x_4 + \eta_{12}\tilde x_3\tilde x_4 = 0,
\end{multline*}
so the element $\tilde x_1\zeta_{234}+\eta_{12}\eta_{34}+
\zeta_{123}\tilde x_4$ is a cocycle in~$A^2$.
 By the definition, one sets the $4$\+tuple Massey product
$\langle x_1,x_2,x_3,x_4\rangle\in H^2(A^\bu)$ to be equal to
the cohomology class of this cocycle.

 Let us briefly describe the tensor version of the quadruple
Massey product.
 Let $K^3\subset K^2\ot H^1\cap H^1\ot K^2\subset H^1(A^\bu)^{\ot 3}$
denote the kernel of the above map~$m_3$.
 Consider the intersection of two vector subspaces
$K^3\ot H^1\cap H^1\ot K^3$ inside $H^1(A^\bu)^{\ot 4}$.
 Then the desired map is
$$
 m_4\:K^3\ot H^1(A^\bu)\cap H^1(A^\bu)\ot K^3\lrarrow
 (H^2(A^\bu)/\im m_2)/\im m_3.
$$ 
 Its explicit construction is based on the same formulas as the above
construction of the 4\+tuple Massey product.
 This is the differential
$$
 d_3^{4,4}\:E_3^{4,4}\lrarrow E_3^{1,2}
$$
in the Eilenberg--Moore spectral sequence.

 The $n$\+ary tensor Massey product of degree-one elements is
a partially defined multivalued linear map
$$
 m_n\:H^1(A^\bu)\ot H^1(A^\bu)\ot\dotsb\ot H^1(A^\bu)
 \ldarrow H^2(A^\bu)
$$
that can be identified (up to a possible plus/minus sign) with
the differential
$$
 d_{n-1}^{n,n}\:E_{n-1}^{n,n}\lrarrow E_{n-1}^{1,2}.
$$
 As in the case of triple Massey products, the constructions of
the $n$\+tuple and $n$\+ary tensor Massey products agree where
the former is defined up to elements of the subspace up to which
the latter is defined.
 However, the domain of definition of the tensor Massey product may be
wider than that of the tuple Massey product, while the tuple Massey
product may produce its outputs with better precision.

\bigskip

 Now let us consider the case when the cohomology algebra $H^*(A^\bu)$
is generated by $H^1$ (as an associative algebra with the conventional
multiplication~$m_2$).
 Then, the map $m_2\:H^1(A^\bu)\ot H^1(A^\bu)\rarrow H^2(A^\bu)$ being
surjective, the above tensor Massey product maps~$m_3$, $m_4$,~\dots\ 
vanish automatically (as their target spaces are zero).
 So do the similar Massey products
$$
 d_{p-1}^{p,\,j_1+\dotsb+j_p}\:
 H^{j_1}(A^\bu)\ot\dotsb\ot H^{j_p}(A^\bu)
 \ldarrow H^{j_1+\dotsb+j_p-p+2}(A^\bu),
$$
$j_1$,~\dots, $j_p\ge1$, \,$p\ge3$, in the higher cohomology.

 Does it mean that all the differentials $d_r^{pq}$ in
the Eilenberg--Moore spectral sequence vanish for $r\ge2$?
 Not necessarily.
 The first possibly nontrivial example would be
$$
 d_2^{4,4}\:H^1(A^\bu)^{\ot 4} \ldarrow
 H^1(A^\bu)\ot H^2(A^\bu)\.\oplus\. H^2(A^\bu)\ot H^1(A^\bu).
$$
 This is the map whose source space is actually the kernel
$\Tor^{H^*(A^\subbu)}_{4,4}(k,k)$ of the differential $d_1^{4,4}\:
H^1(A^\bu)^{\ot 4}\rarrow (H^*(A^\bu)^{\ot 3})^4$, that is, the subspace
$$
 K^2\ot H^1\ot H^1\.\cap\. H^1\ot K^2\ot H^1\.\cap\. H^1\ot H^1\ot K^2
 \.\subset\. H^1(A^\bu)^{\ot 4}
$$
and whose target space is the middle homology space
$\Tor^{H^*(A^\subbu)}_{2,3}(k,k)$ of the sequence
$$
 H^1\ot H^1\ot H^1\lrarrow H^2\ot H^1\.\oplus\.H^1\ot H^2\lrarrow H^3
$$
formed by the differentials $d_1^{3,3}$ and~$d_1^{2,3}$.
 The latter vector space is otherwise known as the space of
relations of degree~$3$ in the graded algebra~$H^*(A^\bu)$.

 What does the map~$d_2^{4,4}$ do?
 Its source space can be otherwise described as the intersection
$$
 (K^2\ot H^1\cap H^1\ot K^2)\ot H^1 \,\cap\,
 H^1\ot (K^2\ot H^1\cap H^1\ot K^2).
$$
 The map $(m_3\ot\id\;\id\ot m_3)$ acts from this subspace to
the quotient space of the vector space $H^2\ot H^1\.\oplus\.H^1\ot H^2$
by the image of the map $(m_2\ot\id\;\id\ot m_2)$ coming from
the direct sum of two copies of $H^1\ot H^1\ot H^1$.
 It is claimed that the map $(m_3\ot\id\;\id\ot m_3)$ can be
naturally lifted to the quotient space of 
$H^2\ot H^1\.\oplus\.H^1\ot H^2$ by the image of only one (diagonal)
copy of $H^1\ot H^1\ot H^1$, as one can see from the explicit
construction of~$m_3$.

 Indeed, let us restrict ourselves to decomposable tensors now
(for simplicity of notation).
 Let $x_1$, $x_2$, $x_3$, $x_4$ be four elements in $H^1(A^\bu)$ for
which $x_1x_2=x_2x_3=x_3x_4=0$ in $H^2(A^\bu)$, and let $\tilde x_1$,
$\tilde x_2$, $\tilde x_3$, $\tilde x_4$ be the liftings of these
elements to $Z^1\simeq H^1$.
 Let $\eta_{12}$, $\eta_{23}$, $\eta_{34}$ be three elements in $A^1$
such that $d\eta_{rs} = \tilde x_r\tilde x_s$ in $B^2\subset A^2$
for all $1\le r<s\le4$, \,$s-r=1$.
 Then the triple Massey products are $\langle x_1,x_2,x_3\rangle =
(\tilde x_1\eta_{23}+\eta_{12}\tilde x_3 \bmod B^2)$ and
$\langle x_2,x_3,x_4\rangle = (\tilde x_2\eta_{34}+\eta_{23}\tilde x_4
\bmod B^2)\in H^2$.
 Replacing $\eta_{rs}$ with $\eta_{rs}'=\eta_{rs}+\tilde y_{rs}$ with
$\tilde y_{rs}\in Z^1$ for all $1\le r<s\le4$, \,$s-r=1$, one obtains
$\langle x_1,x_2,x_3\rangle' = \tilde x_1\eta_{23}'+\eta_{12}'\tilde x_3
\bmod B^2 = \langle x_1,x_2,x_3\rangle + x_1y_{23} + y_{12}x_3$ and
$\langle x_2,x_3,x_4\rangle' = \tilde x_2\eta_{34}'+\eta_{23}'\tilde x_4
\bmod B^2 = \langle x_2,x_3,x_4\rangle + x_2y_{34} + y_{23}x_4$, where
$y_{rs}\in H^1$ are the cohomology classes of the elements
$\tilde y_{rs}\in Z^1$. 
 Finally, one has
\begin{align*}
 (\langle x_1,x_2,x_3\rangle'\ot x_4\;x_1\ot\langle x_2,x_3,x_4\rangle')
 &=(\langle x_1,x_2,x_3\rangle\ot x_4\;x_1\ot\langle x_2,x_3,x_4\rangle)
 \\ &+((x_1y_{23}+y_{12}x_3)\ot x_4\;x_1\ot(x_2y_{34}+y_{23}x_4))
\end{align*}
and
\begin{multline*}
 ((x_1y_{23}+y_{12}x_3)\ot x_4\;x_1\ot(x_2y_{34}+y_{23}x_4)) \\
 = d_1^{3,3}
 (x_1\ot y_{23}\ot x_4 + y_{12}\ot x_3\ot x_4 + x_1\ot x_2\ot y_{34})
\end{multline*}
in $H^2\ot H^1\.\oplus\.H^1\ot H^2$, because
$x_3x_4=x_1x_2=0$ in $H^2(A^\bu)$ by assumption.

 What if the cohomology algebra $H^*(A^\bu)$ is not only generated
by $H^1$, but also defined by quadratic relations?
 There still can be nontrivial tensor Massey operations (i.~e.,
the differentials $d_r^{pq}$ with $r\ge2$), starting from
$$
 d_2^{5,5}\:H^1(A^\bu)^{\ot 5} \ldarrow
 H^2\ot H^1\ot H^1\.\oplus\. H^1\ot H^2\ot H^1\.\oplus\.
 H^1\ot H^1\ot H^2.
$$
 This is actually well-defined as a linear map from the source space
$$\textstyle
 \Tor^{H^*(A^\subbu)}_{5,5}(k,k)\.\simeq\.
 \bigcap_{i=1}^4 H^1(A^\bu)^{\ot i-1}\ot K^2\ot H^1(A^\bu)^{\ot 4-i}
 \.\subset\. H^1(A^\bu)^{\ot 5}
$$
to a target space isomorphic to $\Tor^{H^*(A^\subbu)}_{3,4}(k,k)$.
 The latter Tor space is the first obstruction to Koszulity
of a quadratic graded algebra $H^*(A^\bu)$.

\Section{Noncommutative (Rational) Homotopy Theory}
\label{noncommutative-homotopy-secn}

 From an algebraist's point of view, \emph{rational homotopy theory}
is an equivalence between categories of commutative and Lie
DG\+(co)algebras satisfying appropriate boundedness conditions and
viewed up to quasi-isomorphism.
 The classical formulation~\cite{Quil} claims an equivalence between
the categories of negatively cohomologically graded Lie DG\+algebras
and augmented cocommutative DG\+coalgebras with the augmentation
ideals concentrated in the cohomological degrees~$\le -2$.
 The localizations of such two categories of DG\+(co)algebras by
the classes of (co)multiplicative quasi-isomorphisms are equivalent
over any field of characteristic~$0$; over the field of rational
numbers, these are also identified with the localization of
the category of connected, simply connected topological spaces by
the class of rational equivalences.

 Attempting to include a nontrivial fundamental group into
the picture, people usually consider nilpotent topological spaces,
nilpotent groups, and Malcev completions.
 Here a discrete group is called nilpotent if its lower central
series converges to zero in a finite number of steps.
 Yet the most natural setup for the nilpotency condition is that of
coalgebras rather than algebras, as it allows for infinitary, or
``ind''-conilpotency~\cite{PV,Hin2}.
 Thus it appears that the maximal natural generality for
``an algebraist's version of rational homotopy theory'' is that
of an equivalence between the categories of nonnegatively
cohomologically graded conilpotent Lie DG\+coalgebras and
positively cohomologically graded commutative DG\+algebras,
considered up to quasi-isomorphism over a field
of characteristic~$0$.

 Here the commutative DG\+algebra computes the cohomology algebra
of the would-be topological space, while the conilpotent Lie
DG\+coalgebra is, roughly speaking, dual to the derived rational
completion of its homotopy groups with their Whitehead bracket
(notice the passage to the dual coalgebra in the completion
construction of
Theorem~\ref{pronilpotent-completion-coalgebra-cohomology}
and~\cite[Remark~5.6]{Pbogom}).
 A Lie DG\+coalgebra is called conilpotent if its underlying graded
Lie (super)coalgebra is conilpotent; a nonnegatively graded Lie
supercoalgebra is conilpotent if it is a union of dual coalgebras to
finite-dimensional nilpotent nonpositively graded Lie superalgebras;
and a finite-dimensional nonpositively graded Lie superalgebra is
nilpotent if its degree-zero component is nilpotent Lie algebra and
its action in the components of other degrees is nilpotent.
 (See~\cite[Section~8]{Pbogom} and~\cite[Section~D.6.1]{Psemi} for
a discussion of conilpotent Lie coalgebras and their conilpotent
coenveloping coalgebras.)

 In this section, we make yet another algebraic
generalization/simplification and replace the pair of dual operads
$\mathcal{C}om$-$\mathcal{L}ie$ with that of
$\mathcal{A}ss$-$\mathcal{A}ss$.
 In other words, we consider a noncommutative version of
the above-described theory with commutative DG\+algebras replaced
by associative ones and conilpotent Lie DG\+coalgebras replaced
by conilpotent coassociative ones.
 In this setting, the characteristic~$0$ restriction becomes
unnecessary and one can work over an arbitrary ground field~$k$.
 Thus our aim is to construct an equivalence between the categories
of nonnegatively cohomologically graded conilpotent (coassociative)
DG\+coalgebras and positively cohomologically graded
(associative) DG\+algebras.

\bigskip

 Let $C^\bu$ be a nonzero DG\+coalgebra over a field~$k$; suppose that
it is endowed with a coaugmentation (DG\+coalgebra morphism)
$k\rarrow C^\bu$ and denote the quotient DG\+coalgebra (without counit)
by $C^\bu_+=C^\bu/k$.
 By the definition, the \emph{cobar construction} $\Cob^\bu(C^\bu)$ of
a coaugmented DG\+coalgebra $C^\bu$ is the free associative algebra
$$\textstyle
 \Cob(C)=\bigoplus_{n=0}^\infty C_+[-1]^{\ot n}
$$
generated by the graded vector space $C_+[-1]$ obtained by shifting
by~$-1$ the cohomological grading of the coaugmentation cokernel~$C_+$.
 The grading on $\Cob(C)$ is induced by the grading of~$C_+[-1]$.
 Alternatively, one can define the cobar construction $\Cob(C)$ as
the direct sum of the tensor powers $C_+^{\ot n}$ of the vector space
$C_+$ and endow it with the total grading equal to the sum $i+n$ of
the grading~$i$ induced by the grading of $C_+$ and the grading~$n$
by the number of tensor factors.

 The differential on $\Cob^\bu(C^\bu)$ is the sum of two summands
$d+\d$, the former of them induced by the differential on $C^\bu_+$
and the latter one by the comultiplication in~$C_+^\bu$.
 One has to work out the plus/minus signs in order to make the total
differential on $\Cob^\bu(C^\bu)$ square to zero (cf.\ the definition
of the bar construction in Section~\ref{koszul-implies-quasiform}).

 A quasi-isomorphism of augmented DG\+algebras $A^\bu\rarrow B^\bu$
induces a quasi-isomorphism of their bar constructions
$\Br^\bu(A^\bu)\rarrow\Br^\bu(B^\bu)$ (as one can see from
the filtration and the Eilenberg--Moore spectral sequence
discussed in Section~\ref{koszul-implies-quasiform}).
 However, the morphism of cobar constructions $\Cob^\bu(C^\bu)\rarrow
\Cob^\bu(D^\bu)$ induced by a quasi-isomorphism of coaugmented (even
conilpotent) DG\+coalgebras $C^\bu\rarrow D^\bu$ may not be
a quasi-isomorphism (see Remark~\ref{cobar-construction-counterex}
at the end of this section).
 The reason is that the filtration of the bar construction by the number
of tensor factors is an increasing one, while the similar
filtration of the cobar construction is a decreasing one
(cf.\ the discussion of two kinds of differential Cotor functors
in~\cite{EM1}, \cite{HMS}, and~\cite[Section~0.2.10]{Psemi}).

 As we are interested in the cobar construction as defined above
(i.~e., the direct sum of the tensor powers of $C^\bu_+[-1]$) rather
than its completion by this filtration (which would mean the direct
product of such tensor powers), the related spectral sequence can be
viewed as converging to the cohomology of the cobar construction
only when it is in some sense locally finite.
 This includes two separate cases considered below, which roughly
correspond to the ``conilpotent'' and ``simply connected'' versions
of noncommutative homotopy theory as discussed above.

 The conilpotent version of the theory, which is of primary interest
to us, is based on the following assertion (which does not yet presume
conilpotency, but it will be needed further on).

\begin{prop} \label{nonnegative-dg-coalgebras-cobar-construction}
 Let $C^\bu=(C^0\to C^1\to C^2\to\dotsb)$ and $D^\bu=(D^0\to D^1\to D^2
\to\dotsb)$ be two nonnegatively cohomologically graded coaugmented
DG\+coalgebras.
 Then any comultiplicative quasi-isomorphism $f\:C^\bu\rarrow D^\bu$,
i.~e., a morphism of DG\+coalgebras inducing an isomorphism
$H^*(C^\bu)\simeq H^*(D^\bu)$ of their cohomology coalgebras, induces
a quasi-isomorphism of the cobar constructions
$\Cob^\bu(f)\:\Cob^\bu(C^\bu)\rarrow\Cob^\bu(D^\bu)$.
\end{prop}

\begin{proof}
 For any coaugmented DG\+coalgebra $E^\bu$, set $G^p\Cob^\bu(E^\bu)=
\bigoplus_{n=p}^\infty E_+[-1]^{\ot n}$.
 This is a decreasing filtration of the DG\+algebra $\Cob^\bu(E^\bu)$
compatible with the multiplication and the differential.
 Clearly, a quasi-isomorphism of coaugmented DG\+coalgebras
$C^\bu\rarrow D^\bu$ induces a quasi-isomorphism of the associated
graded algebras $\gr_G\Cob^\bu(C^\bu)\rarrow\gr_G\Cob^\bu(D^\bu)$.
 It remains to observe that when the DG\+coalgebras $C^\bu$ and
$D^\bu$ are nonnegatively cohomologically graded, the filtrations
$G^\bu$ on $\Cob^\bu(C^\bu)$ and $\Cob^\bu(D^\bu)$ are finite at
every (total) cohomological degree.
 Indeed, one has $G^{n+1}\Cob^n(C^\bu)=0=G^{n+1}\Cob^n(D^\bu)$ for
every integer~$n$.
\end{proof}

 A coaugmented DG\+coalgebra $C^\bu$ is called \emph{conilpotent} if
its underlying coaugmented (graded) coalgebra is conilpotent
(see Section~\ref{koszul-implies-kpi1} for the definition).
 The cobar construction $\Cob^\bu(C^\bu)$ of a coaugmented DG\+coalgebra
$C^\bu$ is naturally an augmented DG\+algebra; and
the bar construction $\Br^\bu(A^\bu)$ of an augmented DG\+algebra
$A^\bu$ is a conilpotent DG\+coalgebra.

 The two constructions $C^\bu\longmapsto\Cob^\bu(C^\bu)$ and
$A^\bu\longmapsto\Br^\bu(A^\bu)$, viewed as functors between
the categories of augmented DG\+algebras and conilpotent
DG\+coalgebras, are adjoint functors: for any conilpotent
DG\+coalgebra $C^\bu$ and augmented DG\+algebra $A^\bu$, there is
a bijective correspondence between morphisms of augmented DG\+algebras
$\Cob^\bu(C^\bu)\rarrow A^\bu$ and morphisms of coaugmented
DG\+coalgebras $C^\bu\rarrow\Br^\bu(A^\bu)$.
 The proofs of these results, as well as of the following ones,
can be found, e.~g., in~\cite[Section~6.10]{Pkoszul} (see
also~\cite[Sections~II.3\+4]{HMS} and~\cite[Section~3]{Hin2}).

\begin{prop} \label{bar-cobar-adjunction}
\textup{(a)} For any augmented DG\+algebra $A^\bu$, the adjunction
morphism $\Cob^\bu(\Br^\bu(A^\bu))\rarrow A^\bu$ is a quasi-isomorphism
of augmented DG\+algebras. \par
\textup{(b)} For any conilpotent DG\+coalgebra $C^\bu$,
the adjunction morphism $C^\bu\rarrow\Br^\bu(\Cob^\bu(C^\bu))$ is
a quasi-isomorphism of (conilpotent) DG\+coalgebras.
\end{prop}

 Warning: the assertion of part~(b) does \emph{not} hold without
the conilpotency assumption on~$C^\bu$, and in fact, the adjunction
morphism does not even \emph{exist} without this assumption (cf.\
Remark~\ref{derived-completion} below).

\begin{proof}
 Part~(a): define an increasing filtration on a DG\+algebra $A^\bu$
by the rules $F_0A^\bu=k$ and $F_nA^\bu=A^\bu$ for $n\ge 1$.
 This filtration is compatible with the differential and
the multiplication on $A^\bu$, and so it induces a filtration $F$
compatible with the differential and the comultiplication on
$\Br^\bu(A^\bu)$, and further, a filtration $F$ compatible with
the differential and the multiplication on
$\Cob^\bu(\Br^\bu(A^\bu))$.
 In fact, the filtration on the (co)bar construction induced by
a filtration on a (co)algebra is always compatible with
the tensor (co)multiplication, while compatibility of the original
filtration with the differential and the (co)multiplication ensures
compatibility of the induced filtration with the (co)bar differential.
 The adjunction morphism $\Cob^\bu(\Br^\bu(A^\bu))
\rarrow A^\bu$ is also compatible with the filtrations.

 The passage to the associated graded DG\+(co)algebras with respect
to the filtration $F$ gets rid of all the information about
the multiplication in $A^\bu$; so the DG\+algebra $\gr^FA^\bu$ is
obtained from the DG\+algebra $A^\bu$ by setting the multiplication
on $A_+^\bu$ to be zero, and the DG\+algebra
$\gr^F\Cob^\bu(\Br^\bu(A^\bu))$ is the cobar-construction of
the bar-construction of the DG\+algebra $\gr^FA^\bu$.
 From this point one can proceed further and rid oneself also of
the differential on $A^\bu$ in addition to the multiplication; but
this is unnecessary.
 It suffices to notice that the component of degree~$n$ of
the complex $\gr^F\Cob^\bu(\Br^\bu(A^\bu))$ with respect to
the grading by the indices of the filtration $F$ is the total
complex of a bicomplex composed of $2^{n-1}$ copies of
the complex $A_+^{\ot n}$.
 Proving that such complexes are acyclic for $n\ge 2$ is
elementary combinatorics.

 Part~(b): the argument dual to the one in part~(b) is not immediately
applicable, as the related filtration $G$ on the DG\+coalgebra $C^\bu$
and the induced filtrations on its cobar and bar constructions would
be decreasing ones.
 Instead, consider the canonical increasing filtration
$N_mC^\bu=\ker(C^\bu\to C_+^{\ot m+1})$ on the conilpotent
DG\+coalgebra~$C^\bu$.
 The associated graded DG\+coalgebra $\gr^N\Br^\bu(\Cob^\bu(C^\bu))$ of
the DG\+coalgebra $\Br^\bu(\Cob^\bu(C^\bu))$ by the increasing
filtration induced by the filtration $N$ on $C^\bu$ is identified
with the DG\+coalgebra $\Br^\bu(\Cob^\bu(\gr^NC^\bu))$.
 This reduces the question to the case of the DG\+coalgebra
$\gr^NC^\bu$, which is endowed with an additional positive grading
by the indices of the filtration~$N$.
 Now one endows the DG\+coalgebra $\gr^NC^\bu$ with the decreasing
filtration $G$ with $G^0\gr^NC^\bu=\gr^NC^\bu$, the component
$G^1\gr^NC^\bu$ being the kernel of the counit map $\gr^NC^\bu\rarrow
k$, and $G^n\gr^NC^\bu=0$ for $n\ge2$.
 The induced decreasing filtration $G$ on the DG\+coalgebra
$\Br^\bu(\Cob^\bu(\gr^NC^\bu))$ is locally finite in the grading
by the indices~$m$ of the filtration~$N$.
 This reduces the question to proving that the morphism of
DG\+coalgebras $\gr_G\gr^NC^\bu\rarrow
\Br^\bu(\Cob^\bu(\gr_G\gr^NC^\bu))$ is a quasi-isomorphism, which
can be done by a combinatorial argument similar to the one in part~(a).
\end{proof}

  We recall that an augmented DG\+algebra $A^\bu$ is called
\emph{positively cohomologically graded} if its augmentation ideal
is concentrated in the positive cohomological degrees,
that is $A_+^i=0$ for all $i\le0$.
 Equivalently, a DG\+algebra $A^\bu$ is positively cohomologically
graded if $A^i=0$ for all $i<0$ and $A^0$ is a one-dimensional
vector space generated by the unit element of $A$; any such
DG\+algebra $A^\bu$ has a unique augmentation (DG\+algebra morphism)
$A^\bu\rarrow k$.

\begin{thm} \label{nonhomog-homotopy-theory-equivalence}
 The functors $A^\bu\longmapsto\Br^\bu(A^\bu)$ and $C^\bu\longmapsto
\Cob^\bu(C^\bu)$ induce mutually inverse equivalences between
the category of positively cohomologically graded DG\+algebras
$A^\bu$ with quasi-isomorphisms inverted and the category of
nonnegatively cohomologically graded conilpotent DG\+coalgebras $C^\bu$
with quasi-isomorphisms inverted.
\end{thm}

\begin{proof}
 Having in mind the results of
Propositions~\ref{nonnegative-dg-coalgebras-cobar-construction}
and~\ref{bar-cobar-adjunction}, it suffices to notice that the bar
construction takes positively cohomologically graded augmented
DG\+algebras to nonnegatively cohomologically graded conilpotent
DG\+coalgebras, while the cobar construction takes nonnegatively
cohomologically graded coaugmented DG\+coalge\-bras to positively
cohomologically graded augmented DG\+algebras.
\end{proof}

 For comparison, let us now present the simply connected version of
the theory, that is the direct noncommutative analogue of
``the algebraic part'' of~\cite[Theorem~I]{Quil}.
 A coaugmented DG\+coalgebra $C^\bu$ is called \emph{negatively
cohomologically graded} if its coaugmentation cokernel $C^\bu_+$ is
concentrated in the negative cohomological degrees, that is
$C^i_+=0$ for all $i\ge0$.
 Clearly, any negatively cohomologically graded DG\+coalgebra
is conilpotent.
 Let us call a negatively cohomologically graded DG\+coalgebra
\emph{simply connected} if $C^{-1}=0$, i.~e., $C^i_+=0$ for
all $i\ge-1$.

\begin{prop} \label{simply-connected-dg-coalgebras}
 Any quasi-isomorphism $f\:C^\bu\rarrow D^\bu$ between two simply
connected negatively cohomologically graded DG\+coalgebras
$C^\bu=(\dotsb\to C^{-3}\to C^{-2}\to 0\to k)$ and 
$D^\bu=(\dotsb\to D^{-3}\to D^{-2}\to 0\to k)$ induces
a quasi-isomorphism of the cobar constructions
$\Cob^\bu(f)\:\Cob^\bu(C^\bu)\rarrow\Cob^\bu(D^\bu)$.
\end{prop}

\begin{proof}
 The argument is similar to the proof of
Proposition~\ref{nonnegative-dg-coalgebras-cobar-construction}.
 Once again, one observes that the decreasing filtrations $G$ on
the cobar constructions $\Cob^\bu(C^\bu)$ and $\Cob^\bu(D^\bu)$
are finite at every cohomological degree in our assumptions.
 Indeed, one has $G^{n+1}\Cob^{-n}(C^\bu)=0=G^{n+1}\Cob^{-n}(D^\bu)$
for every integer~$n$.
\end{proof}

 An augmented DG\+algebra $A^\bu$ is called \emph{negatively
cohomologically graded} if one has $A_+^i=0$ for all $i\ge0$.

\begin{thm}
 The functors $A^\bu\longmapsto\Br^\bu(A^\bu)$ and $C^\bu\longmapsto
\Cob^\bu(C^\bu)$ induce mutually inverse equivalences between
the category of negatively cohomologically graded augmented
DG\+algebras $A^\bu$ with quasi-isomorphisms inverted and
the category of simply connected negatively cohomologically graded
DG\+coalgebras $C^\bu$ with quasi-isomorphisms inverted.
\end{thm}

\begin{proof}
 Here one has to notice that the bar construction takes negatively
cohomologically graded augmented DG\+algebras to simply connected
negatively cohomologically graded DG\+coalgebras, while the cobar
construction takes simply connected negatively cohomologically graded
DG\+coalgebras to negatively cohomologically graded augmented
DG\+algebras.
 Otherwise the argument is similar to the proof of
Theorem~\ref{nonhomog-homotopy-theory-equivalence} and based
on the result of Proposition~\ref{simply-connected-dg-coalgebras}.
\end{proof}

\begin{rem}  \label{cobar-construction-counterex}
 The assertions of the above theorems can be modified so as to hold
for arbitrary augmented DG\+algebras and conilpotent DG\+coalgebras.
 One just has to replace the class of quasi-isomorphisms of
DG\+coalgebras with a finer class of \emph{filtered quasi-isomorphisms}
of conilpotent DG\+coalgebras, which are to be inverted in order to
obtain a category equivalent to the category of augmented DG\+algebras
with quasi-isomorphisms inverted (see~\cite[Section~4]{Hin2}
or~\cite[Section~6.10]{Pkoszul}).
 In the form stated above, on the other hand, the assertions of
the theorems do \emph{not} hold already for the negatively
cohomologically graded DG\+coalgebras that are not simply connected---in
fact, this is the class of DG\+coalgebras for which the difference
between quasi-isomorphisms and filtered quasi-isomorphisms becomes
essential.
 It suffices to consider a morphism between two different augmented
$k$\+algebras (viewed as DG\+algebras concentrated in cohomological
degree zero) $A\rarrow B$ inducing an isomorphism
of the Tor spaces $\Tor^A_*(k,k)\simeq\Tor^B_*(k,k)$.
 Then the induced morphism of the bar constructions
$\Br^\bu(A)\rarrow\Br^\bu(B)$ is a quasi-isomorphism of negatively
cohomologically graded conilpotent DG\+coalgebras that is transformed
by the cobar construction into a morphism of DG\+algebras
$\Cob^\bu(\Br^\bu(A))\rarrow\Cob^\bu(\Br^\bu(B))$
with two different cohomology algebras $A$ and~$B$
\cite[Remark~6.10]{Pkoszul}.
\end{rem}

\Section{$K(\pi,1)$-ness + Quasi-Formality Imply Koszulity}
\label{kpi1+quasiform-imply-koszul}

 We refer for the definitions of the bar construction $\Br^\bu(A^\bu)$
of an augmented DG\+algebra $A^\bu$ to
Section~\ref{koszul-implies-quasiform} and of the cobar construction
$\Cob^\bu(C^\bu)$ of a coaugmented DG\+coalgebra $C^\bu$ to
Section~\ref{noncommutative-homotopy-secn}.
 The definition of the cobar construction $\Cob^\bu(D)$ of a coalgebra
$D$ was given previously in Section~\ref{koszul-implies-kpi1};
it is but the particular case of the construction of
Section~\ref{noncommutative-homotopy-secn} corresponding to
the situation of a DG\+coalgebra $C^\bu=D$ concentrated in
the cohomological degree~$0$.
 The definition of the conilpotency property of a coalgebra $C$
can be also found in Section~\ref{koszul-implies-kpi1}.

 The construction of the (tensor) Massey operations on the cohomology
algebra of an augmented DG\+algebra $A^\bu$, understood as
the higher differentials in the algebraic Eilenberg--Moore spectral
sequence (associated with a natural increasing filtration
on the bar-complex $\Br^\bu(A^\bu)$), was introduced and discussed
in Section~\ref{koszul-implies-quasiform}.
 An augmented DG\+algebra $A^\bu$ is called \emph{quasi-formal} if all
these Massey operations vanish.

 Finally, we recall that a graded algebra $H^*$ over a field~$k$
is called \emph{Koszul} if it is concentrated in the positive degrees,
that is $H^i=0$ for $i<0$ and $H^0=k$, and its bigraded Tor coalgebra
(computed by the internally graded DG\+coalgebra $\Br^\bu(H^*)$)
is concentrated in the diagonal grading, i.~e.,
$\Tor^{H^*}_{ij}(k,k)=0$ for $i\ne j$.

 Let $A^\bu$ be an augmented DG\+algebra over a field~$k$.
 We will say that a DG\+algebra $A^\bu$ is \emph{of the $K(\pi,1)$ type}
(or just simply ``a $K(\pi,1)$'') if there exists a conilpotent
coalgebra $C$ over~$k$ such that the DG\+algebra $A^\bu$ can be
connected with the DG\+algebra $\Cob^\bu(C)$ by a chain of
quasi-isomorphisms of augmented DG\+algebras.
 Let us emphasize that, in this definition, the coalgebra $C$,
if viewed as a DG\+coalgebra, must be concentrated in
the cohomological degree~$0$.
 The conilpotency condition on $C$ is of key importance.
 Here the (coassociative and counital, but otherwise arbitrary)
conilpotent coalgebra $C$ plays the role of the conilpotent
coenveloping coalgebra of the conilpotent Lie coalgebra of
a $k$\+complete fundamental group~$\pi$.
 So one could as well write ``$A^\bu=K(C,1)$''.
 (Notice the degree~$1$ homological shift transforming the homotopy
groups of a rational homotopy type endowed with their Whitehead
bracket into a graded Lie (super)algebra.)

 The connection between Koszulity and the Massey operation vanishing was
first pointed out by Priddy~\cite[Section~8]{Pr} (cf.~\cite{LPWZ}),
who was working with algebras endowed with an important positive grading
(in addition to some other possible gradings whose role was to ensure
the locally finite dimension with respect to the multidegree).
 In our language, the result of~\cite[Proposition~8.1]{Pr} can be
reformulated as the following theorem, in which one considers
a coalgebra $C$ endowed with an ``internal'', rather than
a ``cohomological'', grading.
 We refer to~\cite[Section~2]{PV} for the background material about
positively graded coalgebras (see also the next
Section~\ref{self-consistency}).

\begin{thm}  \label{priddy-massey-vanishing}
 Let $C=k\oplus C_1\oplus C_2\oplus C_3\oplus\dotsb$ be a positively
(internally) graded coalgebra cogenerated by its first-degree
component~$C_1$ over a field~$k$.
 Let $A^\bu=\Cob^\bu(C)$ be the cobar DG\+algebra of the coalgebra~$C$
(which is viewed as a DG\+coalgebra concentrated entirely in
the cohomological degree~$0$).
 Then the graded coalgebra $C$ is Koszul if and only if
the differentials $d_{p-1}^{p,q}\:E_{p-1}^{p,q}\rarrow
E_{p-1}^{1,\.q-p+2}$
$$
 (H^*(A^\bu_+)^{\ot p})^q\ldarrow H^{q-p+2}(A^\bu_+)
$$
in the Eilenberg--Moore spectral sequence of the DG\+algebra $A^\bu$
vanish for $p\ge3$.
\end{thm}

\begin{proof}
 By~\cite[Propositions~1 and~2]{PV}, a positively internally graded
coalgebra $C$ cogenerated by its first-degree component is Koszul if
and only if its cohomology algebra $\Ext_C^*(k,k)=H^*(A^\bu)$ is
generated by $\Ext^1_C(k,k)=H^1(A^\bu)$.
 The latter condition does not depend on the internal grading on
the coalgebra~$C$.
 Furthermore, any coalgebra admitting a positive grading
is conilpotent.
 Hence it remains to apply part~(a) of the following proposition.
 (Notice also that, by~\cite[Proposition~3]{PV}, a positively
internally graded coalgebra $C$ cogenerated by its first-degree
component $C_1$ is Koszul if and only if the algebra $H^*(A^\bu)$
is Koszul.)
\end{proof}

\begin{prop}
 Let $A^\bu=\Cob^\bu(C)$ be the cobar DG\+algebra of a conilpotent
coalgebra~$C$ (viewed as a DG\+coalgebra concentrated in
the cohomological degree~$0$).  Then \par
\textup{(a)} the cohomology algebra $H^*(A^\bu)$ is multiplicatively
generated by $H^1(A^\bu)$ if and only if the differentials
$d_{p-1}^{p,q}\:E_{p-1}^{p,q}\rarrow E_{p-1}^{1,\,q-p+2}$
vanish for $p\ge3$; \par
\textup{(b)} the cohomology algebra $H^*(A^\bu)$ is quadratic if and
only if the differentials $d_{p-1}^{p,q}$ as well as the differentials
$d_{p-1}^{p+1,\.q}\:E_{p-1}^{p+1,\.q}\rarrow E_{p-1}^{2,\,q-p+2}$
$$
 (H^*(A^\bu_+)^{\ot p+1})^q\ldarrow
 (H^*(A^\bu_+)\ot H^*(A^\bu_+))^{q-p+2}
$$
vanish for $p\ge3$.
\end{prop}

\begin{proof}
 Part~(a): if the algebra $H^*(A^\bu)$ is multiplicatively generated
by $H^1$, then $E_2^{1,n}=\Tor^{H^*(A^\subbu)}_{1,\.n}(k,k)=0$ for all
$n\ge 2$; since one also has $E_1^{p,q}=(H^*(A^\bu_+)^{\ot p})^q=0$
for all $p>q$, it follows that the differentials $d_{p-1}^{p,q}$
vanish for $p\ge3$.

 Conversely, by Proposition~\ref{bar-cobar-adjunction}(b)
the DG\+coalgebra $\Br^\bu(A^\bu)$ is quasi-isomorphic to~$C$.
 So in the algebraic Eilenberg--Moore spectral sequence $E_r^{pq}$
from Section~\ref{koszul-implies-quasiform} we have
$E_\infty^{p,q}=\gr^F_pH^{p-q}\Br^\bu(A^\bu)=0$ for $p\ne q$,
and in particular $E_\infty^{1,n}=0$ for $n\ge2$.
 This is the situation which people colloquially describe as
``the cohomology $H^*(A^\bu)$ is generated by $H^1(A^\bu)$ using
Massey products''.
 (One should keep in mind that the DG\+algebra $A^\bu=\Cob^\bu(C)$ is
positively cohomologically graded by construction.)
 If all the differentials $d_r^{p,q}$ landing in $E_r^{1,n}$
vanish for $r\ge2$, it follows that $E_2^{1,n}=0$ for $n\ge2$,
so $H^*(A^\bu)$ is generated by~$H^1$ using the conventional
multiplication.

 Part~(b): we can assume that the algebra $H^*(A^\bu)$ is generated
by~$H^1$.
 If this algebra is also quadratic, then $E_2^{2,n}=
\Tor^{H^*(A^\subbu)}_{2,\.n}(k,k)=0$ for all $n\ge3$, so it follows that
the differentials $d_{p-1}^{p+1,\.q}$ vanish for $p\ge3$.
 Conversely, as we explained above, $E_\infty^{2,n}=0$ for $n\ge3$,
so if all the differentials landing in $E_r^{2,n}$ vanish for $r\ge2$,
then we can conclude that $E_2^{2,n}=0$ for $n\ge3$.
\end{proof}

 In the nonhomogeneous conilpotent setting we are working in,
the implication ``$K(\pi,1)$\+ness $+$ quasi-formality imply Koszulity''
becomes a bit more complicated than in
Theorem~\ref{priddy-massey-vanishing}, as the cohomology algebra
$H^*(A^\bu)=H^*\Cob^\bu(C)$ being generated by $H^1$ no longer implies
it being Koszul (cf.\ the final paragraphs of
Section~\ref{koszul-implies-quasiform}).
 The following theorem is the main result of this paper.

\begin{thm} \label{main-theorem}
 The cohomology algebra $H^*=H^*(A^\bu)$ of an augmented DG\+algebra
$A^\bu$ is Koszul if and only if the augmented DG\+algebra $A^\bu$
is simultaneously quasi-formal and of the\/ $K(\pi,1)$ type.
\end{thm}

\begin{proof}
 It was explained in Section~\ref{koszul-implies-quasiform} that
Koszulity of the cohomology algebra $H^*(A^\bu)$ implies vanishing
of the Massey products.
 The assertion that $A^\bu$ is a $K(\pi,1)$ whenever $H^*(A^\bu)$ is
Koszul, announced in the title of Section~\ref{koszul-implies-kpi1},
was not actually proven there (in our present setting) but rather
postponed; so we have to prove it now.
 We start with the following lemma.

\begin{lem}  \label{cofibrant-positively-graded}
 Let $A^\bu$ be an augmented DG\+algebra whose cohomology algebra
$H^*(A^\bu)$ is concentrated in the positive cohomological degrees.
 Then there exists a positively cohomologically graded
DG\+algebra $P^\bu$ together with a quasi-isomorphism
of augmented DG\+algebras $P^\bu\rarrow A^\bu$.
\end{lem}

\begin{proof}
 The construction of a cofibrant resolution of the DG\+algebra $A^\bu$
in the conventional model structure on the category of augmented
DG\+algebras (see~\cite{Hin1}, \cite{Jar},
or~\cite[Section~9.1]{Pkoszul}) provides the desired
DG\+algebra~$P^\bu$.
 One starts from a free graded algebra with generators corresponding
to representative cocycles of a chosen basis in $H^*(A^\bu_+)$, and
then iteratively adds to it new free generators whose differentials
kill the cohomology classes annihilated by the morphism into $A^\bu$.

 The observation is that all the cocycles that need to be killed at
each step, being linear combinations of products of at least two
generators of cohomological degrees~$\ge\nobreak1$, have cohomological
degrees~$\ge 2$.
 So all the new generators that one has to add at this step have
cohomological degrees~$\ge1$.
\end{proof}

 Thus we can assume our DG\+algebra $A^\bu$ to be positively
cohomologically graded; then its bar construction $\Br^\bu(A^\bu)$
is nonnegatively cohomologically graded.
 Now if the cohomology algebra $H^*(A^\bu)$ is Koszul, then it follows
from the Eilenberg--Moore spectral sequence that the cohomology
coalgebra $H^*\Br^\bu(A^\bu)$ of the DG\+coalgebra $\Br^\bu(A^\bu)$
is concentrated in cohomological degree~$0$.

 Hence the embedding $C\rarrow\Br^\bu(A^\bu)$ of the subcoalgebra
$C=\ker(d^0\:\!\Br^0(A^\bu)\to\Br^1(A^\bu))$ of the DG\+coalgebra
$\Br^\bu(A^\bu)$ is a quasi-isomorphism.
 By Proposition~\ref{nonnegative-dg-coalgebras-cobar-construction},
the induced morphism of the cobar constructions
$\Cob^\bu(C)\rarrow\Cob^\bu(\Br^\bu(A^\bu))$ is a quasi-isomorphism,
too.
 By Proposition~\ref{bar-cobar-adjunction}(a),
so is the adjunction morphism $\Cob^\bu(\Br^\bu(A^\bu))\rarrow A^\bu$.
 Finally, the coalgebra $C$ is conilpotent, since its ambient
DG\+coalgebra $\Br^\bu(A^\bu)$ is.
 We have shown that the DG\+algebra $A^\bu$ is a $K(\pi,1)$.

 Now suppose, as the title of this section suggests, that
the augmented DG\+algebra $A^\bu$ is a $K(\pi,1)$ and the Massey
products in its cohomology algebra $H^*(A^\bu)$ vanish.
 Then the augmented DG\+algebra $A^\bu$ is connected by a chain
of quasi-isomorphisms with the DG\+algebra $\Cob^\bu(C)$ for
a certain conilpotent coalgebra~$C$; we can simply assume that
$A^\bu=\Cob^\bu(C)$.
 In particular, the cohomology algebra $H^*(A^\bu)$ is concentrated
in the positive cohomological degrees.

 Applying Proposition~\ref{bar-cobar-adjunction}(b), we can conclude
that the DG\+coalgebra $\Br^\bu(A^\bu)$ is quasi-isomorphic to $C$,
so its cohomology coalgebra $H^*\Br^\bu(A^\bu)$ is concentrated in
cohomological degree~$0$.
 On the other hand, the Massey product vanishing means that one
has $E_2^{pq}=E_\infty^{pq}$ in the Eilenberg--Moore spectral sequence.
 As $E_\infty^{pq}=\gr^F_pH^{q-p}\Br^\bu(A^\bu)=0$ for $p\ne q$,
it follows that $E_2^{pq}=\Tor^{H^*(A^\subbu)}_{pq}(k,k)=0$.
 We have proven that the cohomology algebra $H^*(A^\bu)$ is Koszul.
\end{proof}

\begin{rem} \label{derived-completion}
 Applying the cobar and bar constructions to a nonconilpotent
coaugmented coalgebra $D$ produces a conilpotent DG\+coalgebra
$\Br^\bu(\Cob^\bu(D))$ with the zero-degree cohomology coalgebra
$H^0\Br^\bu(\Cob^\bu(D))$ isomorphic to the maximal conilpotent
subcoalgebra $C=\Nilp D$ of the coaugmented coalgebra~$D$
(see Section~\ref{koszul-implies-kpi1} for the definitions and
notation here and below).
 The DG\+coalgebra $\Br^\bu(\Cob^\bu(D))$ can be called
the DG\+coalgebra of \emph{derived conilpotent completion} of
a coaugmented coalgebra~$D$.
 The DG\+algebra $\Cob^\bu(D)$ is a $K(\pi,1)$ (i.~e., the cohomology
coalgebra of the DG\+coalgebra $\Br^\bu(\Cob^\bu(D))$ is concentrated
in cohomological degree~$0$) if and only if the embedding
$C\rarrow D$ induces a cohomology isomorphism
$\Ext^*_C(k,k)\simeq\Ext^*_D(k,k)$.

 Similarly, applying the cobar and bar constructions to
an augmented algebra $R$ produces a conilpotent DG\+coalgebra
$\Br^\bu(\Cob^\bu(R))$ with the zero cohomology algebra
$H^0\Br^\bu(\Cob^\bu(R))$ isomorphic to the coalgebra of pronilpotent
completion $C=R\comp$ of the augmented DG\+algebra~$R$.
 The DG\+coalgebra $\Br^\bu(\Cob^\bu(R))$ can be called
the DG\+coalgebra of \emph{derived pronilpotent completion} of
an augmented algebra~$R$.
 The DG\+algebra $\Cob^\bu(R)$ is a $K(\pi,1)$ (i.~e.,
the cohomology coalgebra of the DG\+coalgebra $\Br^\bu(\Cob^\bu(R))$
is concentrated in cohomological degree~$0$) if and only if
the natural map of the cohomology algebras $\Ext^*_C(k,k)
\rarrow\Ext^*_R(k,k)$ is an isomorphism.
 These are noncommutative analogues of the procedure of
rational completion of the space $K(\Gamma,1)$ with a discrete
group~$\Gamma$ in rational homotopy theory.

 These observations show, in particular, how to deduce
the assertions of
Theorems~\ref{maximal-conilpotent-subcoalgebra-cohomology}
and~\ref{pronilpotent-completion-coalgebra-cohomology}
from the ``Koszulity implies $K(\pi,1)$\+ness'' claim in
Theorem~\ref{main-theorem}.
\end{rem}

\Section{Self-Consistency of Nonhomogeneous Quadratic Relations}
\label{self-consistency}

 The aim of this section is to explain the thesis, formulated in
the introduction, about the connection between self-consistency
of systems of nonhomogeneous quadratic
relations~\eqref{234} with Koszul principal parts~\eqref{homog2}
and Koszulity of the cohomology algebra $H^*(C)$ of
the coalgebra $C$ defined by such relations.

 Self-consistency of algebraic relations is a fundamental concept
in algebra that is too general to allow a precise general
definition.
 On the other hand, examples of non-self-consistent systems of
relations are easily demonstrated.
 E.~g., consider the following system of nonhomogeneous quadratic
relations of the type~\eqref{210} in two variables $x$ and~$y$:
\begin{equation} \label{inconsistent}
\begin{cases}
 xy-y+1=0, \\
 yx-y=0.
\end{cases}
\end{equation}
 Looking on the relations~\eqref{inconsistent}, one might expect them
to define an ungraded associative algebra of the ``size'' of
the graded algebra with the relations $xy=yx=0$.
 However, proceeding to derive consequences of~\eqref{inconsistent},
one obtains
$$
 (xy)x = (y-1)x = y - x
$$
and
$$
 x(yx) = xy = y - 1
$$
hence $x-1=0$.
 Substituting $x=1$ into the first relation, one comes to $1=0$.
 So the whole algebra defined by~\eqref{inconsistent} vanishes.
 
 Self-consistency of nonhomogeneous quadratic relations of
the type~\eqref{210} is studied in the paper~\cite{Pcurv}
and the book~\cite[Chapter~5]{PP}.
 The main result, called ``the Poincar\'e--Birkhoff--Witt theorem
for nonhomogeneous quadratic algebras'', claims that it suffices
to perform computations with expressions of degree~$\le 3$
when checking self-consistency of relations of the type~\eqref{210}
with Koszul quadratic principal parts~\eqref{homog2}.

 Consequences of non-self-consistency of relations of
the type~\eqref{234} are a bit less dramatic.
 A typical example would be the single relation 
\begin{equation} \label{nonselfconsistent}
 x^2-y^3=0
\end{equation}
for two (noncommutative) variables $x$ and~$y$, implying
$$
 (x^2)x=y^3x\quad\text{and}\quad x(x^2)=xy^3,
$$
hence
\begin{equation} \label{nsc-leads-to}
 xy^3-y^3x=0.
\end{equation}
 No consequences comparable to~\eqref{nsc-leads-to} can be
derived from the quadratic principal part $x^2=0$ \,\eqref{homog2}
of the relation~\eqref{nonselfconsistent}.

 The system of relations
\begin{equation} \label{selfconsistent}
\begin{cases}
 x^2-y^3=0, \\
 xy-yx=0
\end{cases}
\end{equation}
is self-consistent, on the other hand (as we will see below).

 To sum up this informal discussion, one can say that a system
of nonhomogeneous relations~\eqref{210} or~\eqref{234} is
called \emph{self-consistent} if it defines an ungraded (co)algebra
of the same ``size'' as the graded (co)algebra defined by
the homogeneous relations~\eqref{homog2}.
 When the nonhomogeneous relations are not self-consistent,
the ungraded object they define is ``smaller'' than the graded
object defined by the homogeneous principal parts of the relations.
 So self-consistency of relations, considered in this context, appears
as a species or a variation of the notion of a flat deformation.

 In order to formally define self-consistency of
the relations~\eqref{234} in the sense we are interested in,
let us start with the following setup of complete algebras
before moving to coalgebras.
 Let $U$ be a (possibly) infinite-dimensional vector space over
a field~$k$, and let $V=U^*$ be the dual vector space endowed
with its natural locally linearly compact (profinite-dimensional)
topology.
 Consider the cofree conilpotent (tensor) coalgebra
$$
 F = \bigoplus\nolimits_{n=0}^\infty U^{\ot n}
$$
cogenerated by $U$ and the dual topological algebra
$$
 F^*=\prod\nolimits_{n=0}^\infty (U^{\ot n})^* =
 k \.\sqcap\. V \.\sqcap\. V\wot V\.\sqcap\. V\wot V\wot V
 \.\sqcap\.\dotsb
$$
where, by the definition, the completed tensor product of
the dual vector spaces to discrete vector spaces $U'$ and $U''$
is the dual vector space to the tensor product,
$U'{}^*\wot U''{}^*=(U'\ot U'')^*$.
 The conilpotent coalgebra $F$ is endowed with its natural
increasing coaugmentation filtration
$$
 N_mF=\bigoplus\nolimits_{n=0}^m U^{\ot m},
$$
and the topological algebra $F^*$ is endowed with the dual decreasing
augmentation (topological adic) filtration
$$
 N^mF^*=\prod\nolimits_{n=m}^\infty (U^{\ot n})^*=
 (U^{\ot m})^*\.\sqcap\.(U^{\ot m+1})^*\.\sqcap\.\dotsb
$$

 Let $R\subset N^2F^*$ be a closed vector subspace and
$J\subset N^2F^*$ be the closed two-sided ideal in the algebra $F^*$
generated by the subspace~$R$.
 The quotient algebra $A=F^*/J$ is endowed with the quotient topology
and the induced filtration $N^mA=N^mF^*/(N^mF^*\cap J)$.
 The associated graded algebra $\gr_NA$ can be defined as
the infinite product
$$
 \gr_NA=\prod\nolimits_{m=0}^\infty\gr_N^mA=
 \prod\nolimits_{m=0}^\infty N^mA/N^{m+1}A.
$$
 The algebra $\gr_NA$ is the quotient algebra of the algebra
$\gr_NF^*\simeq F^*$ by the closed ideal
$\gr_NJ=\prod_{m=0}^\infty N^mJ/N^{m+1}J$, where
$N^mJ=N^mF^*\cap J$.

 When the space of generating relations $R\subset N^2F^*$ is
homogeneous, that is $R=\prod_n R\cap (U^{\ot n})^*$, the algebra
$A\simeq\gr_NA$ is the product of its grading components,
$A=\prod_nA^n$.
 One has $A^0=k$, \ $A^1=U^*=V$, and
$A^n=(U^{\ot n})^*/R\cap (U^{\ot n})^*$.
 When one can choose the subspace $R$ so that $R\subset V\wot V
\subset N^2F^*$, the algebra $A$ is said to be \emph{defined by}
(\emph{homogeneous}) \emph{quadratic relations}.

 One says that the topological algebra $A$ is \emph{defined by
nonhomogeneous quadratic relations} if the subspace $R\subset N^2F^*$
can be chosen in such a way that the projection map
$R\rarrow N^2F^*/N^3F^*=V\wot V$ is injective.
 The image $\overline R\subset V\wot V$ of this map is
the space of principal quadratic parts~\eqref{homog2} of
the nonhomogeneous quadratic relations~\eqref{234} from~$R$.
 The quotient space $(V\wot V)/\overline R$ is the component
$\gr_N^2A=N^2A/N^3A$ of the associated graded algebra $\gr_NA$,
while the lower grading components are, as above,
$\gr_N^1A=V$ and $\gr_N^0A=k$.

 Denote by $\overline A$ the quotient algebra $F^*/\overline J$
of the algebra $F^*$ by the closed two-sided ideal $\overline J$ 
generated by the subspace $\overline R\subset V\wot V
\subset N^2F^*$.
 Then the identification of the spaces of generators
$\overline A\supset V\simeq\gr_N^1A\subset\gr_NA$ extends uniquely to
a surjective homomorphism of topological graded algebras
\begin{equation} \label{nonhomog-relations-graded-comparison}
 \overline A\lrarrow \gr_NA.
\end{equation}

 The map~\eqref{nonhomog-relations-graded-comparison} is always
an isomorphism in the degrees $n\le3$.
 In particular, one has $\overline A^0=k=\gr_N^0A$, \
$\overline A^1=V=\gr_N^1A$, \ $\overline A^2=(V\wot V)/
\overline R=\gr_N^2A$, and
$$
 \overline A^3=(V\wot V\wot V)/(V\wot\overline R+\overline R\wot V)=
 \gr_N^3A.
$$
 The example of the relation~\eqref{nonselfconsistent} illustrates
how the map~\eqref{nonhomog-relations-graded-comparison} can fail
to be an isomorphism in degree~$4$.
 One says that the system of nonhomogeneous quadratic relations
$R\subset N^2F^*$ is \emph{self-consistent} if the algebra $\gr_NA$
is defined by quadratic relations, or equivalently, if
the map~\eqref{nonhomog-relations-graded-comparison} is an isomorphism
in all the degrees.

 To see that the system of relations~\eqref{selfconsistent} is
self-consistent, one identifies the algebra $A$ defined
by~\eqref{selfconsistent} with the subalgebra in the algebra of
formal power series $k[[z]]$ topologically spanned by the monomials
$z^i$, \ $i=0$ or $i\ge2$, where $x=z^3$ and $y=z^2$.
 The decreasing filtration $N$ on the algebra $A$ is described by
the rules that $N^0A=A$, \ $N^1A$ is spanned by $z^i$, \ $i\ge2$, and
$N^mA=(N^1A)^m$ is spanned by $z^i$, \ $i\ge 2m$.
 One can check that the associated graded algebra $\gr_NA$ is indeed
isomorphic to the quotient algebra of the algebra of noncommutative
formal power series $k\{\{x,y\}\}$ by the closed ideal generated by
the quadratic principal parts~\eqref{homog2}
of the relations~\eqref{selfconsistent}
\begin{equation}
\begin{cases}
 x^2=0, \\
 xy-yx=0,
\end{cases}
\end{equation}
or, which is the same, the quotient algebra of the algebra of power
series $k[[x,y]]$ by the ideal generated by~$x^2$.

 For a more technical discussion of nonhomogeneous relations of
the type~\eqref{234}, the language of conilpotent coalgebras is
preferable.
 Let $C$ be a conilpotent coalgebra over a field~$k$, in
the sense of the definition in Section~\ref{koszul-implies-kpi1}.
 The algebra $A$ above is just the topological algebra dual to $C$,
that is $A=C^*$.

 Let $N_mC=\ker(C\to (C/k)^{\ot m+1})$ denote the canonical increasing
coaugmentation filtration on a conilpotent coalgebra~$C$ (cf.\ the proof
of Proposition~\ref{bar-cobar-adjunction}(b)).
 The filtration $N$ is compatible with the comultiplication on $C$, so
the associated graded vector space $\gr^NC=\bigoplus_m N_mC/N_{m-1}C$
is endowed with a natural structure of positively graded coalgebra.

 We refer to~\cite[Section~2.1]{Pbogom} for the definition of
the cohomology of comodules over a coaugmented coalgebra~$C$.
 In particular, for a left $C$\+comodule $M$, the space
$H^0(C,M)$ is the kernel of the coaction map $M\rarrow C_+\ot_kM$,
where $C_+=C/k$.
 The \emph{cofree} left $C$\+comodule \emph{cogenerated} by
a $k$\+vector space $U$ is the $C$\+comodule $C\ot_kU$.

\begin{lem} \label{cogenerators-comodule}
 Let $M$ be a left $C$\+comodule.  Then \par
\textup{(a)} morphisms of left $C$\+comodules $M\rarrow C\ot_kU$
correspond bijectively to morphisms of $k$\+vector spaces
$M\rarrow U$; \par
\textup{(b)} a morphism of left $C$\+comodules
$f\:M\rarrow C\ot_kU$ is injective if and only if
the morphism $H^0(C,f)\:H^0(C,M)\rarrow H^0(C\;C\ot_kU)=U$
is injective.
\end{lem}

\begin{proof}
 Part~(a) does not depend on the conilpotency assumption on~$C$.
 To a morphism of left $C$\+comodules $f\:M\rarrow C\ot_k U$, one
assigns the composition $M\rarrow C\ot_kU\rarrow U$ of
the morphism~$f$ with the morphism $C\ot_kU\rarrow U$ induced
by the counit map $C\rarrow k$.
 To a morphism of $k$\+vector spaces $g\:M\rarrow U$, one
assigns the composition $M\rarrow C\ot_kM\rarrow C\ot_kU$
of the coaction map $M\rarrow C\ot_kM$ with the map
$C\ot_k g\:C\ot_kM\rarrow C\ot_kU$.

 In part~(b), one has $H^0(C,M)\subset M$ and $H^0(C\;C\ot_kU)
\subset C\ot_k U$, so injectivity of the map~$f$ clearly implies
injectivity of $H^0(C,f)$.
 Conversely, let $L$ denote the kernel of the morphism~$f$; so
$L$ is also a left $C$\+comodule.
 If the map $H^0(C,f)$ is injective, then $H^0(C,L)=0$.
 We have to show that this implies $L=0$.

 Indeed, let $N_mL=\ker(L\to C/N_mC\ot_kL)$ denote the filtration
on $L$ induced by the coaugmentation filtration $N$ on~$C$.
 One has $L=\bigcup_{m=0}^\infty N_mL$.
 The subspace $H^0(C,L)=N_0L\subset L$ is the maximal subcomodule
of $L$ with the trivial coaction of~$C$.
 The subspaces $N_mL\subset L$ are also subcomodules of $L$, and
the coaction of $C$ in the quotient comodules $N_mL/N_{m-1}L$ is
trivial.
 Hence $N_0L=0$ implies by induction $N_mL=0$ for all $m\ge0$
and $L=0$.
\end{proof}

 Lemma~\ref{cogenerators-comodule} purports to explain why
the vector space $U=H^0(C,M)$ is called the \emph{space of
cogenerators} of a left $C$\+comodule~$M$.
 Choosing an arbitrary $k$\+linear map $M\rarrow U$ equal to
the identity isomorphism in restriction to $H^0(C,M)\subset M$
and applying Lemma~\ref{cogenerators-comodule}(a), one obtains
a left $C$\+comodule morphism $M\rarrow C\ot_kU$, which is
injective according to Lemma~\ref{cogenerators-comodule}(b).
 This is the minimal possible way to embed $M$ into a cofree
left $C$\+comodule, in the sense made precise by
Lemma~\ref{cogenerators-comodule}.

 As above, let $F=\bigoplus_{n=0}^\infty U^{\ot n}$ be the tensor
coalgebra of a $k$\+vector space~$U$.

\begin{lem} \label{cogenerators-coalgebra}
 Let $C$ be a conilpotent coalgebra over~$k$.  Then \par
\textup{(a)} morphisms of (coaugmented) coalgebras $C\rarrow F$
correspond bijectively to morphisms of $k$\+vector spaces
$C_+\rarrow U$; \par
\textup{(b)} a morphism of coalgebras $f\:C\rarrow F$ is injective
if and only if the morphism $H^1(f)\: H^1(C)\rarrow H^1(F)=U$
is injective.
\end{lem}

\begin{proof}
 In part~(a) already, it is important that the coalgebra $C$ is
conilpotent.
 It is not difficult to see that any morphism of conilpotent
coalgebras preserves the coaugmentations.
 To a morphism of coalgebras $f\:C\rarrow F$, one assigns
the composition $C_+\rarrow F_+\rarrow U$ of the morphism~$f$ with
the projection $F_+=\bigoplus_{n=1}^\infty U^{\ot n}\rarrow U$.
 To a morphism of $k$\+vector spaces $g\:C_+\rarrow U$, one
assigns the morphism of coalgebras $C\rarrow F$ with
the components $C_+\rarrow U^{\ot n}$ constructed as
the compositions $C_+\rarrow C_+^{\ot n}\rarrow U^{\ot n}$ of
the iterated comultiplication map $C_+\rarrow C_+^{\ot n}$ with
the map $g^{\ot n}\:C_+^{\ot n}\rarrow U^{\ot n}$.
 The assumption of conilpotency of the coalgebra $C$ guarantees
that the map with such components lands inside
$\bigoplus_{n=1}^\infty U^{\ot n}\subset\prod_{n=1}^\infty U^{\ot n}$.

 In part~(b), one has $H^1(C)\subset C_+$ and $H^1(F)\subset F_+$,
so injectivity of the map~$f$ implies injectivity of $H^1(f)$.
 Conversely, let $c\in C$ be an element belonging to $N_nC$ but
not to $N_{n-1}C$, where $n\ge1$.
 Then the image of~$c$ in $C_+^{\ot n}$ is a nonzero element of
the subspace $(N_1C_+)^{\ot n}\subset C_+^{\ot n}$, where
$N_mC_+$ denotes the filtration on the quotient coalgebra without
counit $C_+=C/k$ induced by the filtration $N$ on the coalgebra~$C$.
 In order to show that $f(c)\ne0$ in $F$, it suffices to check that
the map
$$
 (N_1f_+)^{\ot n}\:(N_1C_+)^{\ot n}\rarrow (N_1F_+)^{\ot n}
$$
is injective.
 It remains to recall that $N_1C_+=H^1(C)$ and $N_1F_+=H^1(F)$.
\end{proof}

 Lemma~\ref{cogenerators-coalgebra}(a) purports to explain why
the tensor coalgebra $F$ is called the \emph{cofree conilpotent
coalgebra} cogenerated by a vector space~$U$.
 Furthermore, Lemma~\ref{cogenerators-coalgebra} explains
why the vector space $U=H^1(C)$ is called the \emph{space of
cogenerators} of a conilpotent coalgebra~$C$.
 Choosing an arbitrary $k$\+linear map $C_+\rarrow U$ equal to
the identity isomorphism in restriction to $H^1(C)\subset C_+$
and applying Lemma~\ref{cogenerators-coalgebra}(a), one obtains
a coalgebra morphism $C\rarrow F$ from $C$ into the tensor
coalgebra $F=\bigoplus_{n=0}^\infty U^{\ot n}$, which is injective
by Lemma~\ref{cogenerators-coalgebra}(b).
 This is the minimal possible way to embed $C$ into a tensor
(cofree conilpotent) coalgebra.

 Let $f\:C\rarrow D$ be an injective morphism of conilpotent
coalgebras.
 The quotient space $D/C$ has a natural structure of bicomodule
over the coalgebra $D$, that is a left comodule over the coalgebra
$D\ot_k D^\rop$, where $D^\rop$ denotes the opposite coalgebra to~$D$.
 In other words, the vector space $D/C$ is endowed with natural
left and right coactions $D/C\rarrow D\ot_kD/C$ and $D/C\rarrow
D/C\ot_k D$ of the coalgebra $D$, which commute with each other,
so they can be united in a \emph{bicoaction map}
$$
 D/C\lrarrow D\ot_kD/C\ot_kD.
$$
 Indeed, $D$ is naturally a bicomodule over itself, and the coalgebra
morphism~$f$ endows $C$ with a structure of bicomodule over $D$,
so $D/C$ is the cokernel of a bicomodule morphism $f\:C\rarrow D$.

 We will call the vector space of cogenerators of the bicomodule $D/C$
over $D$ 
$$
 R=R(C,D)=H^0(D\ot_kD^\rop\;D/C)
$$
the \emph{space of defining corelations} of the subcoalgebra $C$ in~$D$.
 This is the dual point of view to constructing the defining relations
of a quotient algebra $A$ of an algebra $B$ as the generators of
the kernel ideal of the morphism $B\rarrow A$.
 The next theorem is the coalgebra version of~\cite[proof
of Proposition~5.2 of Chapter~1]{PP}.

\begin{thm}
 For any injective morphism of conilpotent coalgebras $f\:C\rarrow D$,
there is a natural exact sequence
$$
 0\lrarrow H^1(C)\lrarrow H^1(D)\lrarrow R(C,D)\lrarrow H^2(C)
 \lrarrow H^2(D).
$$
\end{thm}

\begin{proof}
 Consider the short exact sequence of complexes
$$ \dgARROWLENGTH=1.5em
\begin{diagram}
\node{0}\arrow{e}\node{C_+}\arrow{e}\arrow{s}\node{D_+}
\arrow{e}\arrow{s}\node{D/C}\arrow{e}\arrow{s}\node{0} \\
\node{0}\arrow{e}\node{C_+^{\ot2}}\arrow{e}\arrow{s}
\node{D_+^{\ot2}}\arrow{e}\arrow{s}
\node{D_+^{\ot2}/C_+^{\ot2}}\arrow{e}\arrow{s}
\node{0} \\
\node{0}\arrow{e}\node{C_+^{\ot3}}\arrow{e}
\node{D_+^{\ot3}}\arrow{e}
\node{D_+^{\ot3}/C_+^{\ot3}}\arrow{e}
\node{0\makebox[0pt][l]{,}}
\end{diagram}
$$
two of which are just the initial fragments of the cobar-complexes
of $C$ and~$D$.
 The related long exact sequence of cohomology is the desired one,
with the only difference that the kernel of the map $D/C\rarrow
D_+^{\ot 2}/C_+^{\ot 2}$ stands in place of the vector space~$R$.
 To see that these are the same, one first notices that $R$ is
the maximal subbicomodule of $D/C$ where both the left and
the right coactions of $D$ are trivial.
 This can be alternatively constructed as the kernel of the map
$D/C\rarrow D_+\ot_k D/C\oplus D/C\ot_k D_+$ whose components
are the left and the right coaction maps.
 Finally, the natural map $D_+^{\ot 2}/C_+^{\ot 2}\rarrow
D_+\ot_k D/C\oplus D/C\ot_k D_+$ is injective.
\end{proof}

 Assume that the map between the spaces of cogenerators
$H^1(C)\rarrow H^1(D)$ of the coalgebras $C$ and $D$ is an isomorphism.
 Then the space of corelations $R(C,D)$ is identified with the kernel of
the map $H^2(C)\rarrow H^2(D)$.
 In particular, when $D=F$ is a cofree conilpotent coalgebra, we have
$H^2(F)=0$ and $R(C,F)=H^2(C)$.
 This explains why the vector space $H^2(C)$ is called the \emph{space
of defining corelations} of a conilpotent coalgebra $C$ (in a cofree
conilpotent coalgebra).

 The following rule allows to construct a natural filtration $N$ on
the space of corelations $R(C,D)$.
 The components of the coaugmentation filtration $N_mC\subset C$
are subcoalgebras in~$C$.
 Set $N_mR(C,D)=R(N_mC,N_mD)$.
 By assumptions, one has $N_0C=N_0D$ and $N_1C=N_1D$, hence
$N_mR(C,D)=0$ for $m\le1$.
 One has $N_mC=C\cap N_mD$, so the induced morphism $N_mD/N_mC\rarrow
D/C$ is injective, and it follows that the natural morphism
$N_mR(C,D)\rarrow R(C,D)$ is injective, too.
 One has $D/C=\bigcup_m N_mD/N_mC$, hence $R(C,D)=\bigcup_m N_mR(C,D)$.

 We set $N_mH^2(C)=N_mR(C,F)$.
 This construction of a filtration does not depend on the choice
of an embedding of the coalgebra $C$ into a cofree conilpotent
coalgebra $F$ with $H^1(F)=H^1(C)$, because all such embeddings only
differ by an automorphism of the coalgebra~$F$.
 A subcoalgebra $C\subset D$ is said to be \emph{defined by
nonhomogeneous quadratic corelations in $D$} if $R(C,D)=N_2R(C,D)$.
 A coalgebra $C$ is \emph{defined by nonhomogeneous quadratic
corelations} if it is defined by nonhomogeneous quadratic corelations
as a subcoalgebra in~$F$.

\begin{prop}
 For any conilpotent coalgebra $C$, the subspace $N_2H^2(C)\subset
H^2(C)$ coincides with the image of the multiplication
map $H^1(C)\ot_k H^1(C)\rarrow H^2(C)$.
\end{prop}

\begin{proof}
 The vector space $N_2H^2(C)=N_2R(C,F)$ is computed as the cokernel of
the differential $N_2C_+\rarrow L$, where $L$ is the subspace of all
elements in $(N_2C_+)^{\ot2}$ whose images in $(N_2F_+)^{\ot 2}$ are
coboundaries, i.~e., come from elements of $N_2F_+$.
 The image of the differential $N_2F_+\rarrow (N_2F_+)^{\ot2}$ is
equal to $(N_1F_+)^{\ot 2}\subset (N_2F_+)^{\ot 2}$.
 Hence the subspace $L\subset (N_2C_+)^{\ot 2}$ coincides with
$(N_1C_+)^{\ot2}$.
 It remains to recall that $H^1(C)=N_1C_+$.
\end{proof}

\begin{cor}
 A conilpotent coalgebra $C$ is defined by nonhomogeneous quadratic
corelations if and only if the multiplication map
$H^1(C)\ot_k H^1(C)\rarrow H^2(C)$ is surjective.  \qed
\end{cor}

 Let $D=\bigoplus_{n=0}^\infty D_n$, \ $D_0=k$ be a positively graded
coalgebra.
 Then the cohomology spaces $H^i(D)$ are endowed with the internal
grading $H^i(D)=\bigoplus_{j=i}^\infty H^{i,j}(D)$ induced by the grading
of~$D$.
 Assume that $D$ is cogenerated by $D_1$, i.~e., the iterated
comultiplication maps $D_n\rarrow D_1^{\ot n}$ are injective.
 Equivalently, this means that $H^1(D)=H^{1,1}(D)$.
 Then the above-defined filtration $N$ on $H^2(D)$ is associated with
the internal grading, that is $N_mH^2(D)=\bigoplus_{j=2}^mH^{2,j}(D)$.
 In particular, one has $N_2H^2(D)=H^2(D)$ if and only if $H^{2,j}(D)=0$
for $j>2$, that is, if and only if the graded coalgebra $D$ is quadratic.

 A conilpotent coalgebra $C$ is said to be \emph{defined by
self-consistent nonhomogeneous quadratic corelations} if the graded
coalgebra $\gr^NC$ is quadratic.
 According to the main theorem of~\cite{PV}
(see~\cite[Theorem~4.2]{Pbogom} for the relevant formulation), any 
conilpotent coalgebra $C$ with a Koszul cohomology algebra $H^*(C)$
is defined by self-consistent nonhomogeneous quadratic corelations
with Koszul quadratic principal part~\eqref{homog2}.
 Moreover, the seemingly weaker conditions on the algebra $H^*(C)$
formulated in the above Theorem~\ref{with-vishik-main-theorem} are
sufficient in lieu of the Koszulity condition.

 The proof of this result in~\cite{PV} is based on the spectral
sequence converging from $H^*(\gr^NC)$ to $H^*(C)$.
 The above filtration $N$ on $H^2(C)$ is a part of the filtration
$N$ on $H^*(C)$ induced by the filtration $N$ on the cobar-complex
$\Cob^\bu(C)$ induced by the natural filtration $N$ on
a conilpotent coalgebra~$C$.

 The condition of surjectivity of the map $\q H^*(C)\rarrow H^*(C)$
in degree~$2$ in Theorem~\ref{with-vishik-main-theorem} means
that the coalgebra $C$ is defined by nonhomogeneous quadratic
corelations, as we have explained.
 The condition of injectivity of this map in degree~$3$ means,
basically, that syzygies of degree~$3$ between the homogeneous
quadratic parts~\eqref{homog2} of nonhomogeneous quadratic
corelations in $C$ do not lead to non-self-consistencies (as it
happens in the example with the relation~\eqref{nonselfconsistent}).
 In this sense, Theorem~\ref{with-vishik-main-theorem} can be
viewed as an analogue for relations of the type~\eqref{234} of
the Poincar\'e--Birkhoff--Witt theorem for relations
of the type~\eqref{210} \cite{Pcurv,PP} (cf.~\cite{Lee}).

 Conversely, if a conilpotent coalgebra $C$ is defined by
self-consistent nonhomogeneous quadratic corelations with Koszul
quadratic principal parts, that is the graded coalgebra $\gr^NC$
is Koszul, then the algebra $H^*(C)$ is isomorphic to $H^*(\gr^NC)$,
as it easily follows from the same spectral sequence, and
consequently Koszul.

\Section{Koszulity Does Not Imply Formality}
\label{koszul-doesnt-imply-formal}

 Examples of quasi-formal DG\+algebras that are not formal are
well known in the conventional (commutative) rational homotopy
theory~\cite[Examples~8.13]{HS}.
 In this section we, working in the noncommutative homotopy
theory of Section~\ref{noncommutative-homotopy-secn}, over a field
of prime characteristic, present a series of counterexamples of
quasi-formal, nonformal DG\+algebras \emph{with Koszul cohomology
algebras}.
 We also present a family of commutative DG\+algebras with
the similar properties defined over an arbitrary field (of zero or
prime characteristic).

 Recall that a DG\+algebra $A^\bu$ is called \emph{formal} if it
can be connected by a chain of quasi-isomorphisms of DG\+algebras
with its cohomology algebra $H^*(A^\bu)$, viewed as a DG\+algebra
with zero differential (cf.\ Section~\ref{koszul-implies-quasiform}).
 The following lemma shows that there is no ambiguity in this
definition as applied to DG\+algebras with the cohomology algebras
concentrated in the positive cohomological degrees.
 We refer to Section~\ref{noncommutative-homotopy-secn} for a short
discussion of positively cohomologically graded DG\+algebras.

\begin{lem} \label{augmented-quasi}
 Let $A^\bu$ and $B^\bu$ be two augmented DG\+algebras with
the cohomology algebras concentrated in the positive cohomological
degrees, connected by a chain of quasi-isomorphisms of
DG\+algebras over a field~$k$.
 Then there exists a positively cohomologically graded DG\+algebra
$P^\bu$ together with two quasi-isomorphisms of augmented
DG\+algebras $P^\bu\rarrow A^\bu$ and $P^\bu\rarrow B^\bu$.
 In particular, the chain of quasi-isomorphisms between
the DG\+algebras $A^\bu$ and $B^\bu$ can be made to consist of
augmented quasi-isomorphisms of augmented DG\+algebras.
\end{lem}

\begin{proof}
 It suffices to choose a positively cohomologically graded
cofibrant model of either DG\+algebra $A^\bu$ or $B^\bu$ in
the role of $P^\bu$ (see Lemma~\ref{cofibrant-positively-graded}).
\end{proof}

 Recall that any DG\+algebra $A^\bu$ with a Koszul cohomology algebra
$H^*(A^\bu)$ is ``a $K(\pi,1)$'', i.~e., admits a quasi-isomorphism
$\Cob^\bu(C)\rarrow A^\bu$ from the cobar construction of
a conilpotent coalgebra $C$ (see Theorem~\ref{main-theorem} and its
proof).
 The conilpotent coalgebra $C$ can be recovered as the degree-zero
cohomology coalgebra of the bar construction of the DG\+algebra
$A^\bu$, i.~e., $C=H^0\Br^\bu(A^\bu)$.

 As above, let $N_mC=\ker(C\to (C/k)^{\ot m+1})$ be the canonical
increasing filtration on a conilpotent coalgebra~$C$ and let
$\gr^NC=\bigoplus_m N_mC/N_{m-1}C$ be the associated graded coalgebra
(see Section~\ref{self-consistency}).
 The following theorem characterizes those DG\+algebras with
Koszul cohomology algebras that are not only quasi-formal but
actually formal.
 
\begin{thm} \label{formal-graded-thm}
 Let $A^\bu$ be an augmented DG\+algebra with a Koszul
cohomology algebra $H^*(A^\bu)$.
 Then the DG\+algebra $A^\bu$ is formal if and only if
the conilpotent coalgebra $C=H^0\Br(A^\bu)$ is isomorphic to
its associated graded coalgebra\/ $\gr^NC$ with respect to
the canonical increasing filtration~$N$.
\end{thm}

\begin{proof}
 By (the proof of) Theorem~\ref{main-theorem},
the DG\+coalgebra $\Br(A^\bu)$ is quasi-isomorphic to its
degree-zero cohomology coalgebra~$C$.
 The coalgebra $C$ is conilpotent, and its cohomology algebra
$\Ext_C^*(k,k)=H^*\Cob^\bu(C)$, being isomorphic to the algebra
$H^*(A^\bu)$, is Koszul.
 By~\cite[Theorem~4.2]{Pbogom}, it follows that the graded
coalgebra $\gr^NC$ is Koszul and quadratic dual to $H^*(A^\bu)$.
 By the definition of a Koszul graded coalgebra, there is
a natural quasi-isomorphism $\Cob^\bu(\gr^NC)\rarrow H^*(A^\bu)$.

 Hence, whenever the coalgebras $C$ and $\gr^NC$ are isomorphic,
the DG\+algebras $A^\bu$ and $H^*(A^\bu)$ are connected by
a pair of quasi-isomorphisms $\Cob^\bu(C)\rarrow A^\bu$ and
$\Cob^\bu(C)\rarrow H^*(A^\bu)$.
 Conversely, suppose that there is a chain of quasi-isomorphisms
of DG\+algebras connecting $A^\bu$ with $H^*(A^\bu)$.
 By Lemma~\ref{augmented-quasi}, this can be assumed to be
a chain of quasi-isomorphisms of augmented DG\+algebras.
 Applying the bar construction, we obtain a chain of comultiplicative
quasi-isomorphisms connecting the DG\+coalgebras
$\Br^\bu(A^\bu)$ and $\Br^\bu(H^*(A^\bu))$.
 It follows that the degree-zero cohomology coalgebras
$H^0\Br^\bu(A^\bu)=C$ and $H^0\Br^\bu(H^*(A^\bu))=\gr^NC$
of these two DG\+coalgebras are isomorphic.
\end{proof}

 The following series of examples~\cite[Section~9.11]{Partin}
provides a negative answer to a question of Hopkins and
Wickelgren~\cite[Question~1.4]{HW}.

\begin{ex} \label{counter-galois}
 Let $l$ be a prime number and $G$ be a profinite group; denote by
$G^{(l)}$ the maximal quotient pro-$l$-group of~$G$.
 Let $k$ be a field of characteristic~$l$; then the $k$\+vector
space $D=k(G)$ of locally constant $k$\+valued functions on $G$
is endowed with a natural structure of coalgebra over~$k$ with
respect to the convolution comultiplication.
 We will call this coalgebra the \emph{group coalgebra} of
a profinite group $G$ over a field~$k$.
 The maximal conilpotent subcoalgebra $C=\Nilp D\subset D$ is
naturally identified with the group coalgebra $k(G^{(l)})$ of
the pro-$l$-group $G^{(l)}$.
 The cohomology map $H^*(G^{(l)},k)\rarrow H^*(G,k)$ is known to be
an isomorphism, at least, whenever either the cohomology algebra
$H^*(G,k)$ is Koszul~\cite[Corollary~5.5]{Pbogom}, or $G=G_F$ is
the absolute Galois group of a field $F$ containing a primitive
$l$\+root of unity~\cite{Voev}.

 Let $l\ne p$ be two prime numbers and $F$ be a finite extension
of the field of $p$\+adic numbers $\mathbb Q_p$ or the field
of formal Laurent power series $\mathbb F_p((z))$ with coefficients
in the prime field~$\mathbb F_p$.
 Assume that the field $F$ contains a primitive $l$\+root of
unity if $l$~is odd, or a square root of~$-1$ if $l=2$.
 In other words, the cardinality~$q$ of the residue field
$f=\mathcal O_F/\mathfrak m_F$ of the field~$F$ should be such that
$q-1$ is divisible by~$l$ if $l$~is odd and by~$4$ if $l=2$.
 Then the maximal quotient pro-$l$-group $G_F^{(l)}$ of the absolute
Galois group $G_F$ is isomorphic to the semidirect product of two
copies of the group of $l$\+adic integers $\Z_l$ with one of them
acting in the other one by the multiplication with~$q$.

 So, in the exponential notation, the group $H=G_F^{(l)}$ is
generated by two symbols $s$ and $t$ with the relation
$sts^{-1}=t^q$, or, redenoting $s=1+x$ and $t=1+y$ and recalling
that we are working over a field of characteristic~$l$,
\begin{equation}
\begin{aligned}
 (1+x)(1+y)(1+x)^{-1}(1+y)^{-1}&=(1+y^l)^{\frac{q-1}{l}}
 \qquad\text{for $l$ odd, or} \\
 (1+x)(1+y)(1+x)^{-1}(1+y)^{-1}&=(1+y^4)^{\frac{q-1}{4}}
 \qquad\text{for $l=2$.}
\end{aligned}
\end{equation}
 This is a single nonhomogeneous quadratic relation of the
type~\eqref{234} defining the conilpotent group coalgebra
$C=k(H)$.
 The quadratic principal part~\eqref{homog2} of this relation
is simply $xy-yx=0$; this in fact defines the associated graded
coalgebra $\gr^NC$, which turns out to be the symmetric coalgebra
in two variables.

 Alternatively, one can easily compute the cohomology algebra
$H^*(H,k)\simeq H^*(G_F,k)$ to be the exterior algebra in two
generators of degree~$1$; then the graded coalgebra $\gr^NC$ is
recovered as the quadratic dual.
 Either way, the coalgebra $\gr^NC$ is cocommutative and 
the coalgebra $C$ is not (as the group $H$ is not commutative),
so $C$ cannot be isomorphic to $\gr^NC$.
 Applying Theorem~\ref{formal-graded-thm}, we conclude that
the cochain DG\+algebra $\Cob^\bu(C)$ of the pro-$l$-group $H$
is not formal.
 The cochain DG\+algebra $\Cob^\bu(D)=\Cob^\bu(k(G_F))$ of
the absolute Galois group $G_F$, being quasi-isomorphic to
the DG\+algebra $\Cob^\bu(C)$ via the natural quasi-isomorphism
$\Cob^\bu(C)\rarrow\Cob^\bu(D)$ induced by the embedding of
coalgebras $C\rarrow D$, is consequently not formal, either.

 To sum up this example in a more abstract fashion, for any
noncommutative pro-$l$-group $G$ such that $H^*(G,\Z/l)$ is
the exterior algebra over $H^1(G,\Z/l)$ and any field~$k$ of
characteristic~$l$, the cochain DG\+algebra $\Cob^\bu(k(G))$
is not formal.
\end{ex}

 We are not aware of any example of a field $F$ \emph{containing
all the $l$\+power roots of unity} (that is, all the roots of
unity of the powers $l^n$, \,$n\ge1$) whose cochain DG\+algebra
$\Cob^\bu(\Z/l(G_F))$ over the coefficient field~$\Z/l$ is not formal.
 In particular, it would be interesting to know if there is a field
$F$ containing an algebraically closed subfield such that
the DG\+algebra $\Cob^\bu(\Z/l(G_F))$ is not formal for some prime~$l$.
 Our expectation is that such fields do exist, but we cannot pinpoint
any.

 The following family of examples of nonformal commutative
DG\+algebras over an arbitrary field is obtained by
a modification of Example~\ref{counter-galois}.

\begin{ex} \label{counter-lie}
 Consider a single nonhomogeneous quadratic Lie relation of
the type~\eqref{234} for $m+2$ variables~$x$, $y$, $z_1$,~\dots,~$z_m$
\begin{equation} \label{four-variable-lie}
 [x,y]+q_3(z_1,\dotsc,z_m)+q_4(z_1,\dotsc,z_m)+q_5(z_1,\dotsc,z_m)
 +\dotsb=0,
\end{equation}
where $q_n$~are homogeneous Lie expressions of degree~$n$
in the variables~$z_1$,~\dots, $z_m$ over a field~$k$.
 The relation~\eqref{four-variable-lie} can be viewed as defining
a pronilpotent Lie algebra $L$, or its dual conilpotent Lie coalgebra,
or its conilpotent coenveloping coalgebra $C$, or its dual
topological associative algebra, which is simply the quotient algebra
of the algebra of noncommutative formal Taylor power series in
the $m+2$ variables~$x$, $y$, $z_1$,~\dots, $z_m$ by the closed ideal
generated by the single power series~\eqref{four-variable-lie}.

 The homogeneous part~\eqref{homog2} of the
relation~\eqref{four-variable-lie} has the form $[x,y]=0$, and
the relation~\eqref{four-variable-lie} is self-consistent, i.~e.,
the associated graded coalgebra $\gr^NC$ is indeed the conilpotent
coalgebra cogenerated by the $m+2$ variables with the single relation
$xy-yx=0$ (and not a smaller coalgebra).
 One can check this, e.~g., by a trivial application of the
Diamond Lemma~\cite{Berg} for noncommutative power series
(the single relation~\eqref{four-variable-lie} starts with~$xy$,
so there are no ambiguities to resolve).
 The graded coalgebra $\gr^NC$ is Koszul, and its quadratic/Koszul dual
algebra $H^*(C)\simeq H^*(\gr^NC)$ is the connected direct sum of
the exterior algebra in two generators of degree~$1$ and $m$~copies of
the exterior algebra in one generator of degree~$1$.

 Now setting $A^\bu=(\bigwedge(L^*),d)$ to be the Chevalley--Eilenberg
complex of the profinite-dimensional Lie algebra $L$ (i.~e.,
the inductive limit of the Chevalley-Eilenberg cohomological complexes
of the finite-dimensional quotient Lie algebras of $L$ by its open
ideals), one obtains a commutative DG\+algebra endowed with
a natural quasi-isomorphism $\Cob^\bu(C)\rarrow A^\bu$ from
the cobar construction of the coalgebra~$C$.
 The cohomology algebra $H^*(A^\bu)$ is the connected direct sum of
the exterior algebra in two generators and $m$~copies of the exterior
algebra in one generator of degree~$1$, while the bar construction
$\Br^\bu(A^\bu)$ is quasi-isomorphic to~$C$.
 So the commutative DG\+algebra $A^\bu$ cannot be connected with
its cohomology algebra $H^*(A^\bu)$ by a chain of quasi-isomorphisms,
even in the class of noncommutative DG\+algebras, unless
the coalgebra $C$ is isomorphic to $\gr^NC$.
 The latter is easily seen to be impossible, e.~g., when
the degree-three Lie form $q_3(z_1,\dotsc,z_m)$ does not vanish.

 Indeed, consider a variable change~\eqref{123} of the form
$x\to x + \sum_{n\ge2} p_{n,x}(x,y,z)$, \ 
$y\to y + \sum_{n\ge2} p_{n,y}(x,y,z)$, and
$z_i\to z_i + \sum_{n\ge2} p_{n,i}(x,y,z)$, where
$z$~denotes the collection of variables $(z_1,\dotsc,z_m)$
and $\deg p_{n,x}=\deg p_{n,y}=\deg p_{n,i}=n$.
 The relation~\eqref{four-variable-lie} gets transformed by this
variable change into the relation
\begin{multline} \label{substitute} \textstyle
 [x+\sum_n p_{n,x}(x,y,z)\;y+\sum_np_{n,y}(x,y,z)] \\ \textstyle
 + \,q_3(z_1+\sum_np_{n,1}(x,y,z)\;\dotsc\;z_m+\sum_np_{n,m}(x,y,z))
 \, + \,\dotsb\, = 0.
\end{multline}
 The relation~\eqref{substitute} is never equivalent to $[x,y]=0$,
as reducing both of these modulo all the expressions of
degree~$\ge4$ and all the expressions divisible by~$x$ or~$y$,
the former takes the form $q_3(z_1,\dotsc,z_m)=0$, while the latter
vanishes entirely.
 So they cannot generate the same closed ideal in the ring of
noncommutative formal power series.
\end{ex}

\begin{ex} \label{counter-equivalent-relations}
 The following na\"\i ve attempt to construct a commutative version
of Example~\ref{counter-galois} illustrates one of the intricacies
of relations sets~\eqref{234}.
 Consider, instead of the $(m+2)$\+variable Lie relation in
Example~\ref{counter-lie}, a two-variable Lie relation
\begin{equation} \label{two-variable-lie}
 [x,y]+q_3(x,y)+q_4(x,y)+q_5(x,y)+\dotsb=0,
\end{equation}
where $q_n$~are homogeneous Lie expressions of degree~$n$ in
the variables $x$ and~$y$.
 We claim that the relation~\eqref{two-variable-lie} is always
equivalent to the relation $[x,y]=0$ in the world of (Lie or
associative) formal power series in $x$ and~$y$, so the conilpotent
(coenveloping) coalgebra $C$ defined by~\eqref{two-variable-lie}
is in fact cocommutative and isomorphic to $\gr^NC$.
 
 Indeed, the innermost bracket in any Lie monomial in $x$ and $y$
is always $\pm[x,y]$.
 Substituting the expression for $[x,y]$ obtained
from~\eqref{two-variable-lie} in place of the innermost bracket
in every term of degree~$\ge3$ in~\eqref{two-variable-lie}, one deduces
from~\eqref{two-variable-lie} a new Lie relation in $x$ and~$y$
with every term of degree $n\ge3$ replaced by an (infinite) linear
combination of terms of degrees higher than~$n$.
 Continuing in this fashion and passing to the limit in the formal
power series topology, one concludes that
the relation~\eqref{two-variable-lie} implies $[x,y]=0$ (which, in turn,
implies~\eqref{two-variable-lie}).
 E.~g., the relation $[x,y]=[x,[x,y]]$ would imply
$[x,y]=[x,[x,[x,y]]]$, which would lead to $[x,y]=[x,[x,[x,[x,y]]]]$,
etc., and passing to the limit one would finally come to $[x,y]=0$.
\end{ex}
 
 In fact, the exterior algebra in two variables of degree~$1$, which
is the cohomology algebra $H^*(C)$ of the coalgebra $C$ defined
by~\eqref{two-variable-lie}, is a free (super)commutative graded
algebra, so it is intrinsically formal as a commutative graded algebra
(i.~e., any commutative DG\+algebra with such cohomology algebra
is formal).
 Example~\ref{counter-galois} shows that a noncommutative DG\+algebra
with such cohomology algebra does not have to be formal, while
Example~\ref{counter-lie} provides a (super)commutative Koszul graded
algebra that is not intrinsically formal in the commutative world
already.

\end{document}